\documentclass[preprint,12pt]{elsarticle}
\usepackage{amsmath, amssymb, epsfig, graphicx}
\usepackage{paralist}
\usepackage[margin=0.8in]{geometry}
\usepackage{float}
\usepackage{graphicx}
\usepackage{xcolor}
\usepackage[T1]{fontenc}
\usepackage{txfonts}
\usepackage{caption}
\usepackage{subcaption}

\journal{Journal of Mathematical Analysis and Applications}

\def\R{\hbox{{\rm I}\kern - 0.2em{\rm R}\kern 0.2em}}
\def\N{\hbox{{\rm I}\kern - 0.2em{\rm N}\kern 0.2em}}

\newcommand{\BlackBox}{\rule{1.5ex}{1.5ex}}  
\newenvironment{proof}{\par\noindent{\em Proof.\ }}{\hfill\BlackBox\\[2mm]}
\newtheorem{eg}{Example}[section]

\newtheorem{thrm}[eg]{Theorem}
\newtheorem{lem}{Lemma}

\newtheorem{prop}{Proposition}

\def \eqref#1{(\ref{#1})}
\def \Arctan{\mathop{\rm Arctan}\nolimits}

\newcommand{\dss}{\displaystyle}
\def\lgt{\stackrel{<}{\scriptstyle >}}
\def\glt{\stackrel{>}{\scriptstyle <}}

\begin{document}
\begin{frontmatter}
\title{Stability in a scalar differential equation with\\ multiple, distributed time delays}
\author{Sue Ann Campbell\corref{cor1}} 
\ead{sacampbell@uwaterloo.ca}
\address{Department of Applied Mathematics and Centre for Theoretical Neuroscience,\\
University of Waterloo, Waterloo, Ontario, N2L 3G1, Canada}
\cortext[cor1]{Corresponding author}
\author{Israel Ncube\corref{cor2}}
\ead{Ncube.Israel@gmail.com}
\address{
College of Engineering, Technology $\&$ Physical Sciences,
Department of Mathematics,\\
Alabama A $\&$ M University,
4900 Meridian Street N.,
Huntsville, AL 35762, U.S.A.}

\begin{abstract}
We consider a linear scalar delay differential equation (DDE), consisting of two 
arbitrary distributed time delays. We formulate necessary conditions for stability 
of the trivial solution which are independent of the distributions. For the case of
one discrete and one gamma distributed delay, we give an explicit description of
the region of stabiltiy of the trivial solution and discuss how this depends on
the model parameters.
\end{abstract}
\begin{keyword}
delay differential equations \sep distributed delay \sep stability boundaries
\MSC 34K20 
\end{keyword}

\end{frontmatter}

\section{Introduction}\label{model_description}
Distributed time delays arise in models for a variety of applications including
population dynamics \cite{Cush_1,Far_Ol,KCP,Kuang93}, blood cell dynamics \cite{BBM,YuanBelair}, neuronal models \cite{JessopCampbell,TSE}, and coupled oscillators \cite{Atay03a,Atay03b}.
Although many of these models include only a single time delay, this often results from some
simplification in the model set up. As an example, it is common to assume that all of the time delays are identical
\cite{JessopCampbell} or to neglect relatively smaller time delays \cite{KCP}. The stability of equilibria in models with two discrete time delays has been studied extensively \cite{Bela_Camp_1, Hale_Huang_1, Mah_etal_1, Yuan_Campbell04}.
It has been shown that the presence of two time delays can lead to phenomena such as
stability switching and the existence of codimension two bifurcation points
\cite{Bela_Camp_1, Yuan_Campbell04}. In this article, we investigate such phenomena in a model with two distributed time delays. We focus our attention on the following scalar delay differential equation (DDE) with a linear decay:
\begin{equation}
\dot x(t)= -kx(t)+\alpha\int_0^\infty x(t-\tau)f_\alpha(\tau)\,d\tau +\beta\int_0^\infty x(t-\tau)f_\beta(\tau)\,d\tau\;,
\label{main_eqn}
\end{equation}
where $k,\;\alpha,\; \beta$ are real numbers, $f_\alpha(T)$ and $f_\beta(T)$
are {\em arbitrary} distributions, satisfying
\[
\int_{0}^{\infty}f_{\alpha}(s)ds=1=\int_{0}^{\infty}f_{\beta}(s)ds\;.
\]

We note that \eqref{main_eqn} is a delay differential equation with infinite delay. 
Thus the appropriate phase space is
$C_{0,\rho}((-\infty,0],\R)$ where $\rho$ is a positive constant \cite{Diek_Gyll,Hino,Kol_Mys_2}.
This is the Banach space of functions $\psi:(-\infty,0]\rightarrow\R$ such that
$e^{\,\rho\,\theta}\psi(\theta)$ is continuous and 
\[\lim_{\theta\rightarrow\infty} e^{\,\rho\,\theta}\psi(\theta)=0,\] with norm 
$\|\psi\|_{\infty,\rho}=\sup_{\theta\le 0} e^{\,\rho\,\theta}\psi(\theta)$.  In this space, 
we need the following additional conditions on the distributions
\[
\int_{0}^{\infty}e^{\,\rho\, s}f_{\alpha}(s)\,ds<\infty,\quad 
\int_{0}^{\infty}e^{\,\rho\, s}f_{\beta}(s)\,ds<\infty.
\]
Since equation
\eqref{main_eqn} is linear it will have a unique solution for any initial function 
$\phi\in C_{0,\rho}((-\infty,0],\R)$ \cite{Diek_Gyll,Hino,Kol_Mys_2}. 

Stability with general distributions has been studies by some authors, but generally
only with a single delay \cite{Arino_Vdd06,BBM,Ruan06,WXR,WXW}. 
In particular we note the work of Anderson \cite{And91,And92} which studies
stability properties of linear, scalar differential equation with a single distributed
time delay, in terms of the moments of the distribution.  

Stability in the presence of multiple distributed delays has been studied in some models, 
generally by fixing the distributions to some specific form \cite{And93,incube2013a,incube2013b,LWW,RF}.  An exception is the work of Faria and Oliveira \cite{Far_Ol} which studies the global 
stability of equilibria in a class of 
Lotka-Volterra models with multiple distributed delays having finite maximum delay.  They give 
conditions on the interaction coefficients of the system which guarantee asymptotic
stability for any distribution.

Various specific time delay kernels have been used in the literature. The two most commonly used ones are the {\em weak} and the {\em strong} kernels (gamma distributions), given by $f(s)=re^{-rs}$ and $f(s)=r^2 se^{-rs}$, with $r>0$, respectively. It is well-known (see \cite{MacD_1} and \cite{MacD_2}, for instance) that the average time delays associated with the weak and the strong delay kernels are given by $T=\frac{1}{r}$ and $T=\frac{2}{r}$, respectively. Equation \eqref{main_eqn} would occur in the linearisation about an equilibrium point for the models of \cite{Bela_Camp_1,Hale_Huang_1} and \cite{Yuan_Campbell04,incube2013b} if the discrete delays were replaced by distributed delays.

Making the change of variables $\widetilde{x}=x,\;\widetilde{t}=kt$, and
defining new parameters by $\widetilde{\alpha}=\frac{\alpha}{k}$,
$\widetilde{\beta}=\frac{\beta}{k}$, $T=k\tau$, and new distributions
\[g_\alpha(T)=\frac{1}{k}f_\alpha\left(\frac{T}{k}\right),\quad g_\beta(T)=\frac{1}{k}f_\beta\left(\frac{T}{k}\right),\]
we rescale \eqref{main_eqn} to get
\begin{equation}
\dot x(t)= -x(t)+\alpha\int_0^\infty x(t-T)g_\alpha(T)\,dT +\beta\int_0^\infty x(t-T)g_\beta(T)\,dT\;,
\label{scaled_eqn}
\end{equation}
where, for notational tractability, we have dropped the tilde's.  

In this paper, we will investigate the stability of the trivial solution of \eqref{scaled_eqn} by adopting direct analysis of the associated characteristic equation. The paper is organised as follows. In Section~\ref{delay-independent-stability}, we formulate some necessary distribution-independent conditions for stability of the trivial solution. In Section~\ref{generation-mechanism}, we describe some distribution-specific mechanisms by which bifurcation curves evolve in an appropriate parameter space, and how this has an effect on the region of stability. Section~\ref{number-of-roots-with-real-part} discusses how stability changes as bifurcation curves are traversed in the parameter space characterised by $\beta$ and
the mean delay of $g_\beta$.

\section{Distribution-independent stability}\label{delay-independent-stability}
In equation \eqref{scaled_eqn}, we make the ansatz that $x(t) \simeq c e^{\lambda t}\;,\;c\in\mathbb{R}\;,\;\lambda\in\mathbb{C}$, to obtain the associated characteristic equation, which is given by
\begin{equation}
D(\lambda):=\lambda + 1 -\alpha\int_0^\infty e^{-\lambda T}g_\alpha(T)\,dT
-\beta\int_0^\infty e^{-\lambda T}g_\beta(T)\,dT=0.
\label{char_eq}
\end{equation}
It is well-known \cite{Diek_Gyll} that the trivial solution of \eqref{scaled_eqn}
will be asymptotically stable if all the roots of the characteristic equation
have negative real parts and unstable if at least one root has a positive real
part. In this section we focus on deriving conditions for stability and instability 
which do not depend on the particular distributions that occur in the
equation.

\begin{thrm}\label{instabthm}
If $\alpha+\beta>1$, then the trivial solution of \eqref{scaled_eqn}
is unstable.
\end{thrm}
\begin{proof}
Assume that $\lambda$ is a real root of \eqref{char_eq}. Then we have
\begin{eqnarray*}
D(\lambda)&=&\lambda + 1 -\alpha\int_0^\infty e^{-\lambda T}g_\alpha(T)\,dT-\beta\int_0^\infty e^{-\lambda T}g_\beta(T)\,dT \\
&\ge & \lambda + 1 - |\alpha|\int_0^\infty g_\alpha(T)\,dT - |\beta|\int_0^\infty g_\beta(T)\,dT \\
&\ge &\lambda + 1 -|\alpha| -|\beta|.
\end{eqnarray*}
Consequently, for $\lambda$ real and sufficiently large, we conclude that $D(\lambda)>0$. Furthermore, we note that
\[D(0)=1 -\alpha\int_0^\infty g_\alpha(T)\,dT -\beta\int_0^\infty g_\beta(T)\,dT
=1-(\alpha+\beta) <0.
\]
Thus, since $D(\lambda)$ is continuous, we conclude that it has a root with positive real part.
\end{proof}

Let $\tau_{\alpha}$ and $\tau_{\beta}$ be the mean delays of $g_\alpha$ and $g_\beta$, respectively. That is,

\[ \tau_{\alpha}= \int_0^\infty T g_\alpha(T)\,dT,\qquad
\tau_{\beta}= \int_0^\infty T g_\beta(T)\,dT. \]
\begin{thrm} \label{rouche}
Assume that $D(\lambda)$ is analytic in $Re(\lambda)>-d$ for some $d>0$. The
trivial solution of \eqref{scaled_eqn} is asymptotically stable
if $|\alpha|+|\beta|< 1$.
\end{thrm}
\begin{proof}
We will prove this result by the use of Rouch\'e's Theorem
\cite[p.\ 313]{Chur_Brow}. To begin, let
\[
f_1(\lambda)=\lambda+1-\alpha\int_0^\infty e^{-\lambda T}g_\alpha(T)\,dT
\qquad
f_2(\lambda)=-\beta\int_0^\infty e^{-\lambda T}g_\beta(T)\,dT,
\]
and consider the contour in the complex plane, $C=C_1 \cup C_2$, given by
\[ C_1:\ \lambda=Re^{i\theta},\ -\frac{\pi}{2}\le \theta\le\frac{\pi}{2}
\qquad C_2:\ \lambda=i y,\ -R\le y\le R,\]
where $\dss R\in\mathbb{R}$. On $C_1$, we have that
\begin{eqnarray*}
|f_2(\lambda)|&=& |-\beta \int_0^\infty e^{-RTe^{i\theta}}g_\alpha(T)\,dT|\\
&\le &|\beta|\int_0^\infty e^{-RT\cos(\theta)}|e^{-iRT\sin(\theta)}|g_\alpha(T)\,dT\\
&\le &|\beta|\int_0^\infty g_\alpha(T)\,dT\\
&=& |\beta|.
\end{eqnarray*}
Furthermore, we note that $f_1(\lambda)=Re^{i\theta}+1-\alpha G_\alpha $,
where
\begin{eqnarray*}
G_\alpha&=&
\int_0^\infty \left[\cos(RT\sin(\theta))+i\sin(RT\sin(\theta)) \right]
e^{-RT\cos(\theta)}g_\alpha(T)\,dT\\
&=& G^R_\alpha+iG^I_\alpha.
\end{eqnarray*}
Hence, we obtain
\begin{eqnarray*}
|f_1(\lambda)| &=& \sqrt{(R\cos(\theta)+1+G_\alpha^R)^2+(R\sin(\theta)+G_\alpha^I)^2}\\
&=& \sqrt{R^2+2R\cos(\theta)+1+|G_\alpha|^2+2G_\alpha^R+
2R[\cos(\theta)G_\alpha^R+\sin(\theta)G_\alpha^I]} \\
&=& \sqrt{R^2+2R\cos(\theta)+1+|G_\alpha|^2+2G_\alpha^R+
2R\int_0^\infty \cos(\theta+RT\sin(\theta))e^{-RT\cos(\theta)} g_\alpha(T)\,dT}\\
&\ge &\sqrt{R^2+1-2 -2R}\\
&=& \sqrt{(R-1)^2-2}.
\end{eqnarray*}
Thus, for $R$ sufficiently large, $|f_1(\lambda)| > |f_2(\lambda)|$ on $C_1$. On $C_2$, we have that
\begin{eqnarray*}
|f_2(\lambda)|&=& |-\beta \int_0^\infty e^{-iyT}g_\alpha(T)\,dT|\\
&\le & |\beta| \int_0^\infty g_\alpha(T)\,dT\\
&= & |\beta|.
\end{eqnarray*}
Additionally, it is crucial to note that if $|\alpha|<1$, then
\begin{eqnarray*}
|f_1(\lambda)| &=& |iy+1 -\alpha\int_0^\infty e^{-iyT} g_\alpha(T)\,dT|\\
&=& \sqrt{(1-\alpha\int_0^\infty \cos(yT) g_\alpha(T)\,dT)^2
+(y+\alpha\int_0^\infty \sin(yT) g_\alpha(T)\,dT)^2} \\
&\ge & \sqrt{(1-|\alpha|)^2}\\
&=& 1-|\alpha|.
\end{eqnarray*}
As a result, if $1-|\alpha|>|\beta|$, then $|f_1(\lambda)|>|f_2(\lambda)|$ on $C_2$.
We note that if $\alpha\ne 0$, and $\beta\ne 0$, then both $f_1$ and $f_2$ do not reduce to zero anywhere on $C$. Thus, by Rouch\'e's Theorem, if
$1-|\alpha|>|\beta|>0$ and $R$ is sufficiently large, then $f_1(\lambda)$ and
$D(\lambda)=f_1(\lambda)+f_2(\lambda)$ have the same number of zeroes
inside $C$. In the limit as $R\rightarrow\infty$, it is easy to see that $f_1(\lambda)$ and $D(\lambda)$
have the same number of zeroes with $\rm{Re}(\lambda)>0$.
It has been shown \cite{CampbellJessop09} that all the zeroes of $f_1(\lambda)$
have negative real part if $|\alpha|<1$.
This completes the proof.
\end{proof}
For our final result, we specialise to the situation in which one of the time delays is
discrete. That is,
\begin{equation}
 g_\alpha(T)=\delta(T-\tau_\alpha).
\label{galpha_disc}
\end{equation}
We begin with the following.
\begin{lem}\label{Ssignlem}
The function $\omega+\xi\sin(\omega \tau_\alpha)>0$ for all $\omega>0$ if, and
only if,
\begin{enumerate}
\item $-\frac{1}{\tau_\alpha}\le \xi\le 0$,
\item $0\le \xi\le \frac{u^*}{\tau_\alpha}$,
where $u^*\approx 4.603$ is the unique positive zero of
\begin{equation}
2\pi-\cos^{-1}\left(-\frac{1}{u}\right)+u\sin\left(2\pi-\cos^{-1}\left[-\frac{1}{u}\right]\right).
\label{ustareq}
\end{equation}
\end{enumerate}
\end{lem}
\begin{proof}
Let $S(\omega)=\omega+\xi\sin(\omega \tau_\alpha)$. Clearly, $S(0)=0$ and
$\displaystyle\lim_{\omega\rightarrow\infty}S(\omega)>0$. Now,
$\frac{dS}{d\omega}= 1+\xi \tau_\alpha \cos(\omega \tau_\alpha)$,
which is clearly positive
for all $\omega>0$ if $|\xi \tau_\alpha| < 1$. It follows that $S(\omega)>0$
for all $\omega>0$ if $|\xi \tau_\alpha| < 1$.

If $\xi\tau_\alpha=-1$ then $S'(0)=S''(0)=0$, $S'''(0)>0$ and
$S'(\omega)\ge 0$ for $\omega>0$. Thus, $S(\omega)>0$ for all $\omega>0$ in
this case as well. Finally, we note that if $\xi\tau_\alpha<-1$, then
$\left.\frac{dS}{d\omega}\right|_{\omega=0}<0$. The first result follows.

If $\xi>0$, however, $\left. \frac{dS}{d\omega}\right|_{\omega=0}>0$
for any value of $\tau_\alpha$. $S(\omega)$ can change sign only if it is
zero for some $\omega$. $S(\omega)$ will first become zero when there exists
$\xi,\tau_\alpha$, and $\omega^*>0$ such that
$S(\omega^*)=0$ and $S'(\omega^*)=0$. In other words, when
\begin{eqnarray*}
\omega^*+\xi\sin(\omega^* \tau_\alpha)&=&0\;,\\
1+\xi \tau_\alpha\cos(\omega^* \tau_\alpha) &=& 0\;.
\end{eqnarray*}
A simple rearrangement shows that these equations are equivalent to \eqref{ustareq},
where $u=\xi \tau_\alpha$. It is straightforward to show that
\eqref{ustareq} has a unique positive zero, $u^*$, and that if
$\xi>u^*/\tau_\alpha$, then there exists an $\omega>0$ such that $S(\omega)<0$.
\end{proof}
We are now in a position to state our final result.
\begin{thrm} \label{rouche2}
Let $g_\alpha$ be given by \eqref{galpha_disc} and assume that
$D(\lambda)$ in \eqref{char_eq} is analytic in $Re(\lambda)>-d$ for some $d>0$.
Then the trivial solution of \eqref{scaled_eqn} is asymptotically stable if
$\alpha<-1$, $\tau_\alpha\le -1/(2\alpha)$, and $|\beta|<-\alpha-1$.
\end{thrm}
\begin{proof}
The proof setup is identical to that of Theorem~\ref{rouche}, except that
$f_1=\lambda+1-\alpha e^{-\lambda \tau_\alpha}$ in the present case. The proof is the same
except that on $C_2$ we have
\begin{eqnarray*}
|f_1(\lambda)| &=&
\sqrt{1+\alpha^2-2\alpha\cos(y\tau_\alpha)+y(y+2\alpha\sin(y\tau_\alpha))}\\
&\ge & \sqrt{(1-|\alpha|)^2}\\
&=& -\alpha-1\;,
\end{eqnarray*}
where we have used Lemma~\ref{Ssignlem} to show that the term
$y(y+2\alpha\sin(y\tau_\alpha))$ is non-negative if
$\tau_\alpha\le -1/(2\alpha)$. We note
that all of the roots of $f_1(\lambda)$ have negative real parts if
	\[ \alpha<-1 \quad  \mbox{and} \quad \tau_\alpha\le -\frac{1}{2\alpha}\]
\cite{CampbellJessop09}.
The rest of the proof proceeds in a manner analogous to the proof of Theorem~\ref{rouche}.
\end{proof}
The results of this section are illustrated in Figure~\ref{instabfig}.
\begin{figure}[t!]
\centering
\begin{subfigure}{0.49\textwidth}
\centering
\caption{$\alpha>1$}
\includegraphics[scale=.4]{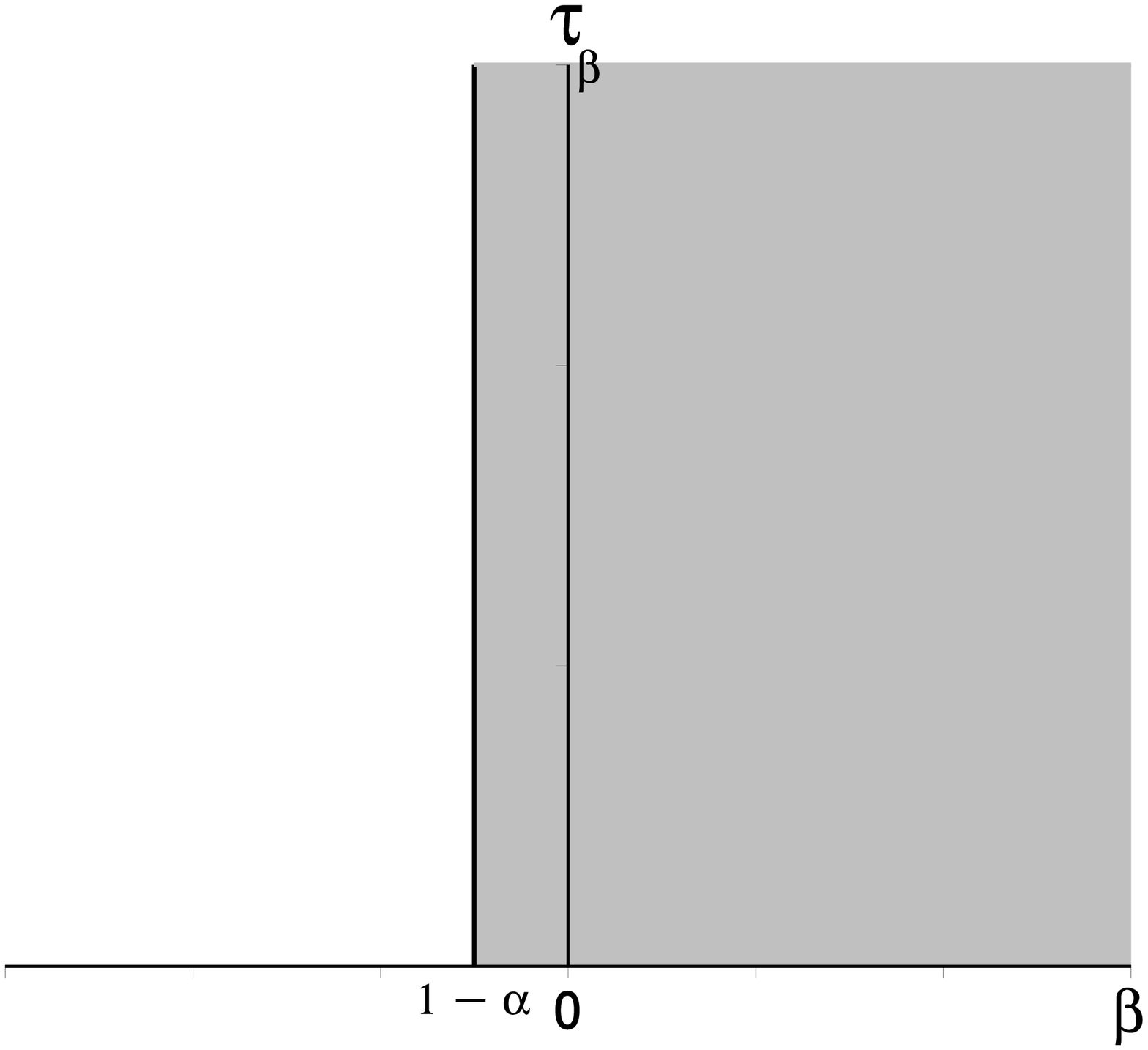}
\end{subfigure}
\begin{subfigure}{0.49\textwidth}
\centering
\caption{$0<\alpha<1$}
\includegraphics[scale=.4]{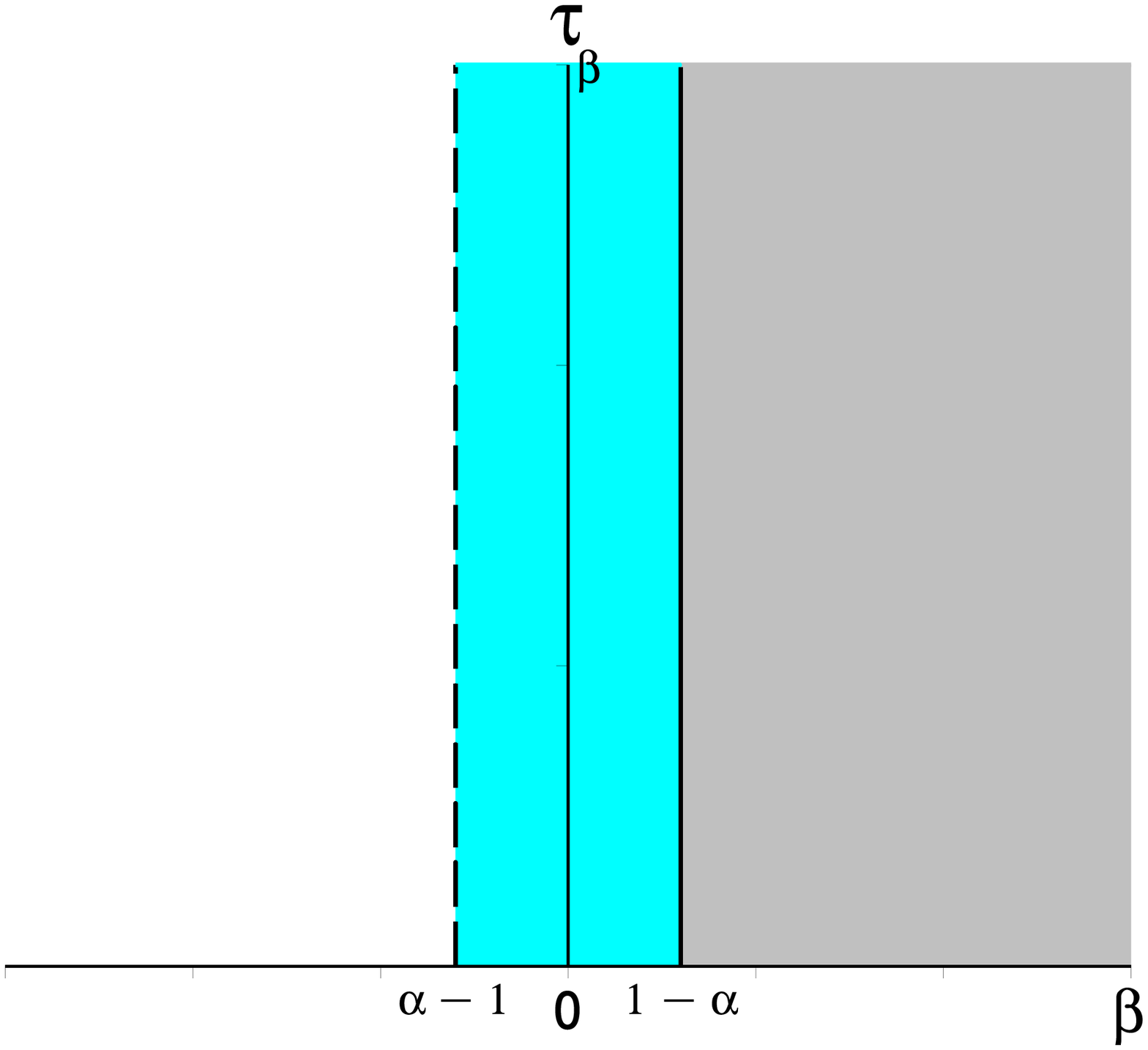}
\end{subfigure}
\begin{subfigure}{0.49\textwidth}
\centering
\caption{$-1<\alpha<0$}
\includegraphics[scale=.4]{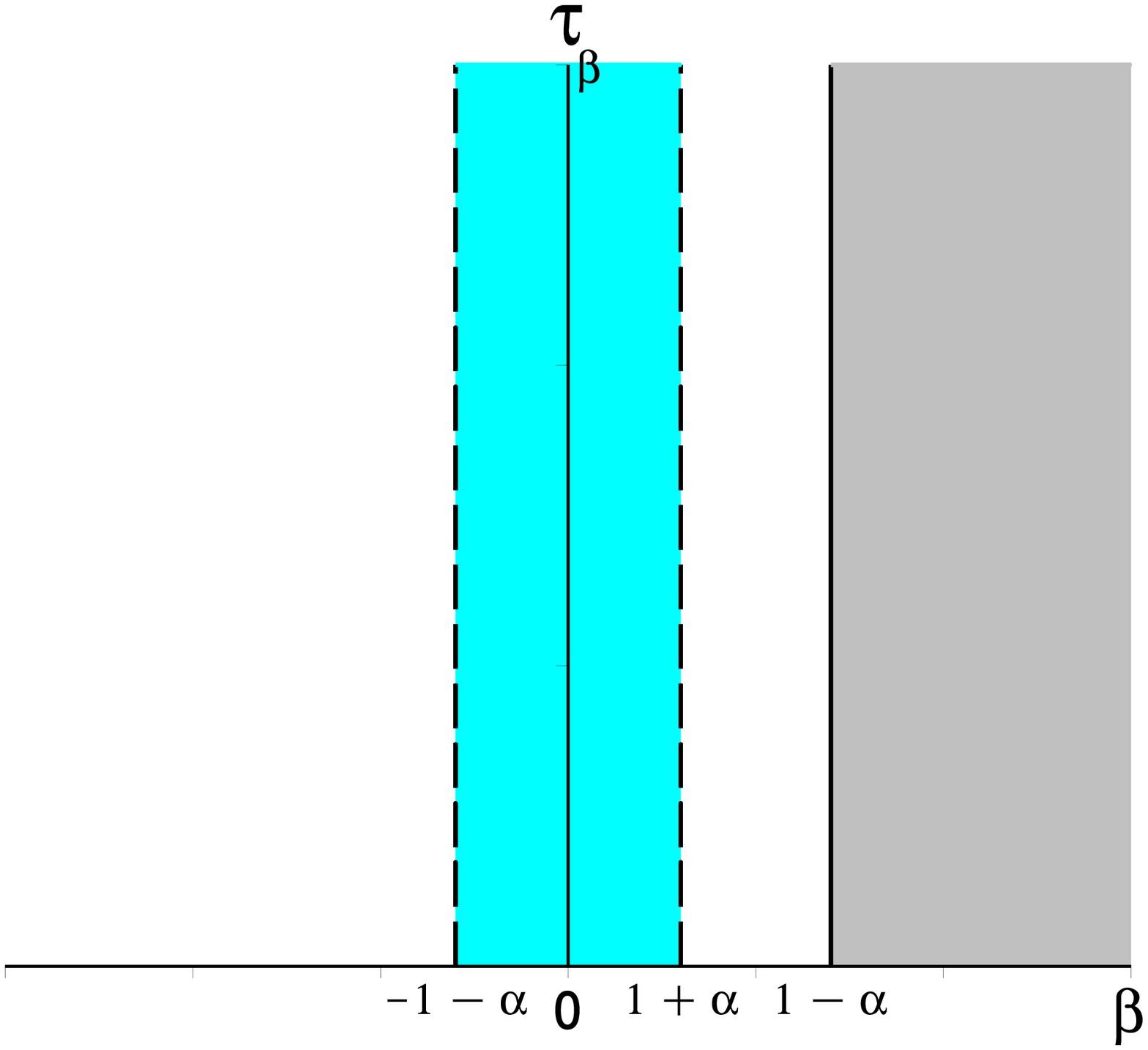}
\end{subfigure}
\begin{subfigure}{0.49\textwidth}
\centering
\caption{$\alpha<-1$ $\tau_{\alpha}<-1/(2\alpha)$}
\includegraphics[scale=.4]{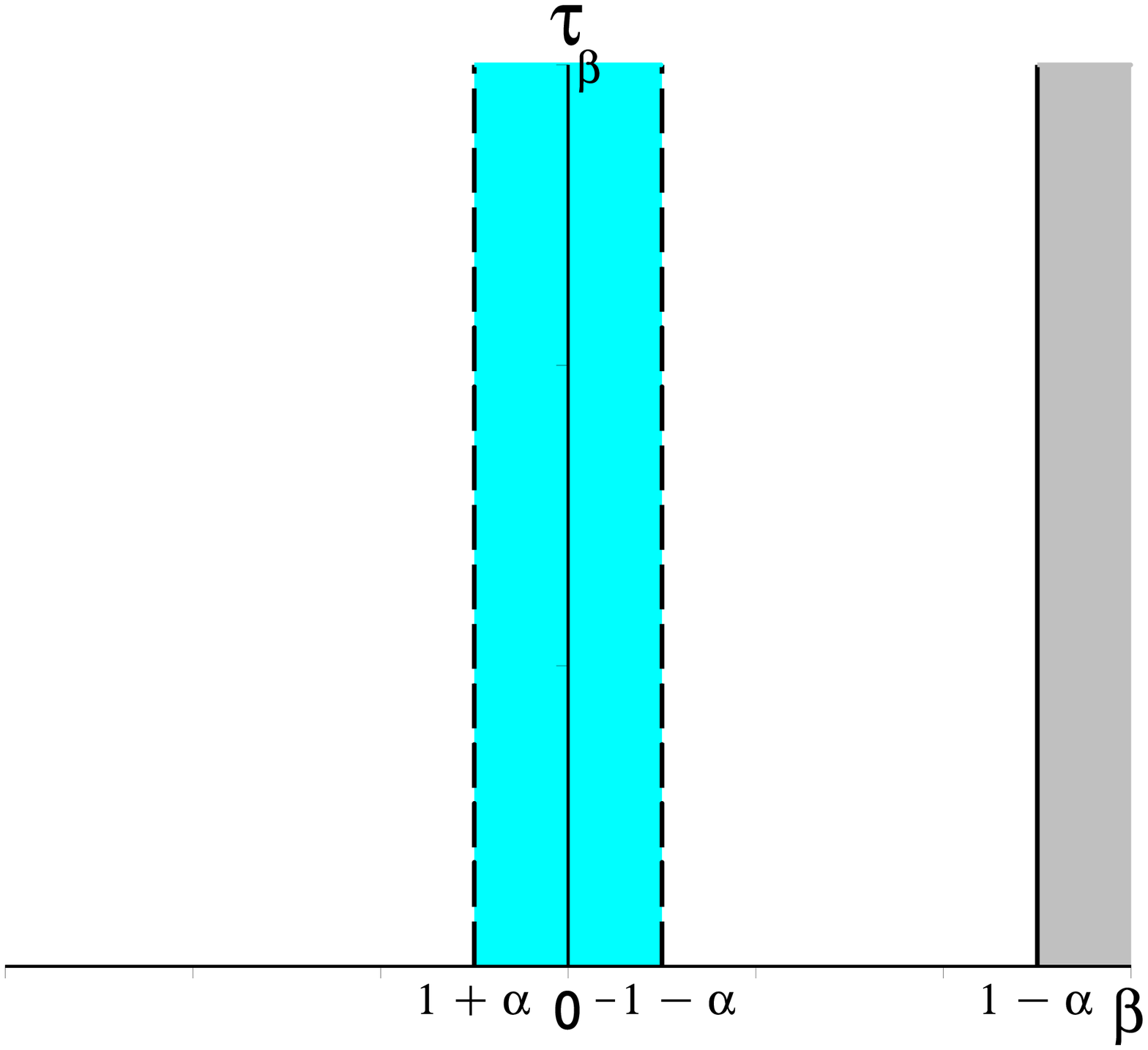}
\end{subfigure}
\caption{Illustration of results in the $(\beta, \tau_\beta)$  plane. Regions
shaded in grey correspond to the region of instability described in
Theorem~\ref{instabthm}. Regions shaded in cyan correspond to the regions of
stability described in Theorem~\ref{rouche} (b),(c), and Theorem~\ref{rouche2} (d).}
\label{instabfig}
\end{figure}
The unshaded regions of Figure~\ref{instabfig} are regions in which we do not know whether
the trivial solution is stable or not. To investigate this further, we need
more information about the distributions. This is the focus of Section~\ref{generation-mechanism}.
\section{Stability Boundaries}\label{generation-mechanism}
The last section described regions in parameter space where the trivial
solution is unstable and regions where it is asymptotically stable regardless
of the choice of distribution. However, these regions do not cover the
entire parameter space; indeed, there are regions where the stability is unknown.
This issue will be addressed in the rest of the paper. We will attempt to give a concrete description of the
complete stability region in the $(\beta,\tau_\beta)$ parameter space,
for the following choices of distributions:
\[ g_\alpha(T)=\delta(T-\tau_\alpha), \]
\begin{equation}
\displaystyle
g_\beta(T)= \frac{a^m T^{m-1} e^{-aT}}{(m-1)!}\;,\;a>0\;,\;m>0\;,
\label{gamma_distr}
\end{equation}
so the first time delay is discrete and the second is a gamma 
distribution.
With these choices, equation \eqref{scaled_eqn} becomes
\begin{equation}
\dot{x}(t) = -x(t) + \alpha x(t-\tau_\alpha) + \frac{\beta a^m}{(m-1)!}\int_{0}^{\infty}x(t-T)T^{m-1}e^{-aT}dT\;,
\label{rescaled_eqn}
\end{equation}
and the mean delay of $g_\beta$ is
\begin{equation}
\tau_\beta \stackrel{def}{=} \frac{m}{a}.
\label{tau_beta}
\end{equation}
Note that $\tau_\beta>0$, since $m>0$ and $a>0$.
As a consequence, the characteristic equation \eqref{char_eq} becomes
\begin{equation}
D(\lambda)=\lambda+1-\alpha e^{-\lambda\tau_\alpha} - \beta\frac{a^m}{(\lambda+a)^m} = 0.
\label{char_eqn}
\end{equation}
It is important to note that $D(\lambda)$ is analytic in $Re(\lambda)>-a$.

The boundary of the stability region will be determined by curves where
the characteristic equation has one or more eigenvalues with zero real
part. The main goal of this section is to give a concrete description of these curves.
To begin, we consider the case of a zero eigenvalue. Setting $\lambda=0$
in equation \eqref{char_eqn} shows that this will occur when $\beta=1-\alpha$.

The curves where pure imaginary eigenvalues occur are more complicated.
In equation \eqref{char_eqn}, we set $\lambda=i\omega$, $\omega>0$,
to obtain
\begin{equation}
i\omega+1-\alpha e^{-i\omega \tau_\alpha} - \beta \frac{a^m}{(i\omega + a)^m} = 0\;.
\label{complex_eqn}
\end{equation}
We define
\begin{equation}
\tan\theta \stackrel{def}{=} \frac{\omega}{a}\;,
\label{tantheta}
\end{equation}
and note that $\omega>0,a>0\ (\tau_\beta>0)$ imply that
$\theta\in(2k\pi,(2k+\frac{1}{2})\pi),\ k\in\mathbb{Z}^+\cup\{0\}$.
Substituting this into equation \eqref{complex_eqn},
rearranging, and decomposing into real and imaginary components, we arrive at the system:
\begin{eqnarray}
1 - \alpha \cos(\omega \tau_\alpha) - [\omega + \alpha\sin(\omega \tau_\alpha)]\tan (m\theta) - \beta r &=& 0\;,\nonumber\\
\omega + \alpha\sin(\omega \tau_\alpha) + [1-\alpha \cos(\omega \tau_\alpha)]\tan(m\theta) &=& 0\;,
\label{bif_surface}
\end{eqnarray}
where $r \stackrel{def}{=} \frac{\cos^m \theta}{\cos m\theta}$. The second
equation of system \eqref{bif_surface} yields
\begin{equation}
\begin{array}{lcl}
\tan(m\theta) &=& -\frac{\omega + \alpha\sin(\omega \tau_\alpha)}{1-\alpha\cos(\omega \tau_\alpha)}\\
&\stackrel{def}{=}& h(\omega)\;,
\end{array}
\label{tanmtheta}
\end{equation}
which we substitute back into the first equation in \eqref{bif_surface} to arrive at
\begin{equation}
\beta \cos^m \theta =
\frac{[(\omega +\alpha\sin(\omega \tau_\alpha))^2 + (1-\alpha\cos(\omega \tau_\alpha))^2]}{(1-\alpha\cos(\omega \tau_\alpha))}\,\cos(m\theta)\;.
\label{beta}
\end{equation}
The relation in \eqref{bif_surface} is manipulated to obtain
\begin{equation}
\beta =
\pm \sqrt{(1-\alpha\cos(\omega \tau_\alpha))^2+(\omega+\alpha\sin(\omega \tau_\alpha))^2}\, \cos^{-m} \theta\;,
\label{beta_expr}
\end{equation}
which is then substituted back into equation \eqref{beta}, giving
\begin{equation}
\cos(m\theta)
=
\pm\frac{(1-\alpha\cos(\omega \tau_\alpha))}{\sqrt{(1-\alpha\cos(\omega \tau_\alpha))^2+(\omega+\alpha\sin(\omega \tau_\alpha))^2}}\;.
\label{cosmtheta}
\end{equation}
From \eqref{tau_beta} and \eqref{tantheta}, we have that
\begin{equation}
\tau_\beta = \frac{m}{\omega}\tan\theta\;.
\label{mean_delay}
\end{equation}
Equation \eqref{tanmtheta} gives the following representation of $\theta$
\begin{equation}
m\theta
= \left\{
\begin{array}{lcl} \Arctan[h(\omega)]+2l\pi\;&,&\;\mbox{if}\;\; \cos(m\theta) \ge 0\;, \\
                            \Arctan[h(\omega)]+(2l+1)\pi\;&,&\;\mbox{if}\;\; \cos(m\theta) < 0\;.
\end{array}
\right.
\label{mtheta}
\end{equation}
where $\displaystyle l=0,1,2,\cdots$.
Employing \eqref{cosmtheta} and \eqref{mtheta} in
\eqref{beta_expr} and \eqref{mean_delay}
yields expressions for $\beta$ and $\tau_\beta$ in terms of
$(\omega,l)$ and the model parameters. Note that we must
restrict $\omega$ so that $\tau_\beta>0$. From \eqref{beta_expr}, \eqref{mtheta}, and \eqref{mean_delay}, it follows that the 2-tuple
$(\beta(\omega, l),\tau_\beta(\omega,l))$ may be `periodic'; depending on the values of the indices $m>0$ and $l\in\mathbb{Z}^{+}\cup\{0\}$. As an example, for a given $m\in\mathbb{Z}^+$, the parametric curve
$(\beta(\omega, l),\tau_\beta(\omega,l))$ replicates at an $l$ for which: (i) $\frac{2l+1}{m}=2$, if $\cos(m\theta)<0$; and (ii) $\frac{2l}{m}=2$, if $\cos(m\theta)>0$.
In Subsections~\ref{uniqueness-of-curves} and \ref{replication-of-curves}, we will attempt to characterise the dynamics of the total
number of curves in the $(\beta(\omega,l),\tau_\beta(\omega,l))$ parameter space for $m\in\mathbb{Z}^{+}$ and a restricted set of parameter values.
\subsection{How unique curves are generated in the $(\beta,\tau_\beta)$ parameter space}\label{uniqueness-of-curves}
This section examines how {\em unique} 2-tuples $(\beta,\tau_\beta)$ are generated in the parameter space, as the gamma distribution index
$m\in\mathbb{Z}^{+}$ is varied. These curves represent stability boundaries in the parameter space, and so an analysis of how they evolve is essential
in understanding stability properties of the trivial solution of \eqref{rescaled_eqn}.
Let us begin by recalling the results of Lemma~\ref{Ssignlem} and noting that the denominator
of \eqref{tanmtheta} satisfies $1-\alpha \cos(\omega \tau_\alpha) \ge 0$ if $|\alpha|\le 1$.
It is clear that if $0<\alpha\le \min(1,u^*/\tau_\alpha)$ or
$\max(-1,-1/\tau_\alpha)\le \alpha<0$,
then $h(\omega) \in (-\infty, 0]$ and $\displaystyle \lim_{\omega \rightarrow 0}h(\omega) = 0$. This leads
to $\displaystyle \frac{1}{m}\Arctan[h(\omega)] \in \left(-\frac{\pi}{2m},0\right]$. In the following description, we consider two cases, namely:
$\displaystyle \cos(m\theta)>0$ and $\displaystyle \cos(m\theta)<0$. Since $|\alpha|<1$, the sign of $\cos(m\theta)$ is determined by
the choice of $\pm$ in \eqref{cosmtheta}. As a consequence, when $\cos(m\theta)>0$, we have that
\begin{equation}
\displaystyle
\begin{array}{ccl}
\beta^{+}(\omega,l^{+}) &=&
\sqrt{[1-\alpha\cos(\omega \tau_\alpha)]^2 + [\omega+\alpha\sin(\omega \tau_\alpha)]^2}\,\cos^{-m}\theta^{+}(\omega,l^{+})\;,\\
\tau_\beta^{+}(\omega,l^{+}) &=& \frac{m}{\omega}\tan\theta^{+}(\omega,l^{+})\;,\\
\theta^{+}(\omega,l^{+}) &=& \frac{\Arctan [h(\omega)]}{m} + \frac{2l^{+}\pi}{m}\;,
\end{array}
\label{positive_cos}
\end{equation}
where we have used the fact that $\cos^m(\theta)>0$. From this,
we observe that $\tau_\beta^{+}(\omega,0) = \frac{m}{\omega}\tan(\frac{1}{m}\Arctan[h(\omega)]) < 0$ for all $m \in \mathbb{Z}^{+}$. This means that we
need to restrict our consideration to $l^+>0.$

To understand how unique curves are generated in parameter space, for \eqref{positive_cos}, we need to determine when
$\displaystyle (\beta^{+}(\omega,l^{+}), \tau_\beta^{+}(\omega,l^{+})) \neq (\beta^{+}(\omega,l^{+}+p), \tau_\beta^{+}(\omega,l^{+}+p))$, where
$p\in \mathbb{Z}^{+}$ is to be determined. If we consider $m$ odd, we see that $\displaystyle \cos^{m}\theta^{+}(\omega,l^{+})$
is $2\pi$-periodic in $\theta$ since 
\[ \cos^{m}\theta^{+}(\omega,l^{+}) = \frac{1}{2^n}\left[1+\cos 2\theta^{+}(\omega,l^{+})\right]^n\, \cos\theta^{+}(\omega,l^{+}),\] 
where $n\in\mathbb{Z}^{+}\cup\{0\}$.
Now, if we set $p=m$, we note that $\displaystyle \theta^{+}(\omega,l^{+}+m)=\theta^{+}(\omega,l^{+})+2\pi$, which implies that $\displaystyle \cos^{m}\theta^{+}(\omega,l^{+})
=\cos^{m}\theta^{+}(\omega,l^{+}+m)$ and $\displaystyle \tan\theta^{+}(\omega,l^{+}) = \tan\theta^{+}(\omega,l^{+}+m)$. From this, we obtain that
$\displaystyle \left(\beta^{+}(\omega,l^{+}),\tau_\beta^{+}(\omega,l^{+})\right)=\left(\beta^{+}(\omega,l^{+}+m),\tau_\beta^{+}(\omega,l^{+}+m)\right)$,
which means that the parametric curves for $l^{+}$ and $l^{+}+m$ are identical.  We conclude that parametric curves of the form
$\displaystyle \left(\beta^{+}(\omega,l^+),\tau_\beta^{+}(\omega,l^+)\right)$ are distinct for $l^+=1,2,\cdots m-1$, with $m$ fixed. The above analysis is easily extended to the case $m$ even. Here we see that $\displaystyle \cos^{m}\theta^{+}(\omega,l^{+}) = \frac{1}{2^k}\left[1+\cos 2\theta^{+}(\omega,l^{+})\right]^k$, which is clearly $\pi$-periodic in $\theta$. Furthermore, we note that $\displaystyle \theta^{+}(\omega,l^{+}+\frac{m}{2})
= \theta^{+}(\omega,l^{+}) + \pi$, implying that
$\displaystyle \left(\beta^{+}(\omega,l^{+}),\tau_\beta^{+}(\omega,l^{+})\right)= \left(\beta^{+}(\omega,l^{+}+\frac{m}{2}),\tau_\beta^{+}(\omega,l^{+}+\frac{m}{2})\right)$. Therefore, for $m$ even, curves of the form
$\left(\beta^{+}(\omega,l^+),\tau_\beta^{+}(\omega,l^+)\right)$ are distinct for $l^+=1,2,\cdots, \frac{m}{2}-1$. We now look at the case $\displaystyle \cos (m\theta) < 0$, for which
\begin{equation}
\displaystyle
\begin{array}{ccl}
\beta^{-}(\omega,l^{-}) &=& -\sqrt{[1-\alpha\cos(\omega \tau_\alpha)]^2+[\omega + \alpha\sin(\omega \tau_\alpha)]^2}\,\cos^{-m}\theta^{-}(\omega,l^{-})\;,\\
\tau_\beta^{-}(\omega,l^{-}) &=& \frac{m}{\omega}\tan\theta^{-}(\omega,l^{-})\;,\\
\theta^{-}(\omega,l^{-}) &=& \frac{\Arctan[h(\omega)]}{m} + \frac{(2l^{-}+1)\pi}{m}\;.
\end{array}
\label{negative_cos}
\end{equation}
By an argument analogous to the one above, we arrive at the following results. For $m$ odd, the curves
$\displaystyle \left(\beta^{-}(\omega,l^-),\;\tau_\beta^{-}(\omega,l^-)\right)$
are distinct for $\displaystyle l^{-}=1,2,\cdots,m-1$. Similarly, for $m$ even, we get that the curves
$\displaystyle \left(\beta^{-}(\omega,l^-),\;\tau_\beta^{-}(\omega,l^-)\right)$ are distinct for $l^-=1,2,\cdots,\frac{m}{2}-1$. In conclusion, we note that the total number of distinct curves
$(\beta^{-}(\omega,l^{-}),\tau_\beta^{-}(\omega,l^{-}))$ and $(\beta^{+}(\omega,l^{+}),\tau_\beta^{+}(\omega,l^{+}))$ in parameter space
is influenced only by the index $m$. It is also clear that the total number of such curves in parameter space increases as $m$ is increased.

\subsection{Generation of new curves in the $(\beta,\tau_\beta)$ parameter space}\label{replication-of-curves}
As $m$ increases, new curves are `born' in the $(\beta,\tau_\beta)$
parameter space, as suggested by the preceding discussion. Thus, as $m$ increases, so does the number of parametric curves. Recalling that $\tau_\beta$ is restricted to be positive and $\cos^m(\theta)>0$ shows that the curves
$(\beta^+(\omega,l^+),\tau_\beta^+(\omega,l^+))$ lie in the first quadrant,
while
$(\beta^-(\omega,l^-),\tau_\beta^-(\omega,l^-))$ lie in the second quadrant.
It remains to determine how or whether $l^+,l^-$ can be chosen so that
 $\tau_\beta^+,\tau_\beta^->0$. That is, such that
$\theta^+(\omega,l^+),\theta^+(\omega,l^-) \in (2k\pi,(2k+\frac{1}{2})\pi)$
for some $k\in\mathbb{Z}^+\cup\{0\}$.

We begin by recalling that $\frac{1}{m}\Arctan[h(\omega)] \in (-\frac{\pi}{2m},0)$.
From this, it is evident that
$\theta^{+}(\omega,l^+)=\frac{1}{m}\Arctan[h(\omega)] + \frac{2l^{+}\pi}{m} \in
\frac{\pi}{m}\left(\frac{4l^{+}-1}{2},2l^{+}\right)$.
Thus, for $\theta^{+}$ to overlap in the correct interval (so that $\tau_\beta>0$), we must
have integers $m,l^+$, and $k$ such that
\[ 2k\pi\le \frac{4l^+-1}{2m}\pi < (2k+\frac{1}{2})\pi,\quad
\mbox{or}\quad
 2k\pi< \frac{2l^+}{m}\pi \le (2k+\frac{1}{2})\pi. \]
The first of these inequalities is equivalent to
\[ mk+\frac{1}{4}\le l^+ < mk+\frac{1}{4}+\frac{m}{4},\]
from which we conclude that $m\ge 4$ in order for this to be possible.
Further, by recalling that $l^+=0,..,m-1$, we find that we must have $k=0$
to satisfy the inequality. A similar analysis of the second inequality
yields the same restrictions. Setting $k=0$, we can determine which values
of $l^+$ generate appropriate values of $\theta^+$. This is summarised
in Table~\ref{lvalues}.
\begin{table}
\begin{center}
\begin{tabular}{|c|c|c|}
\hline
$m$& $l^+$&$l^-$\\
\hline
$1$ & none & none\\
\hline
$2,3$ & none & 0\\
\hline
$4,5$ & $1$ & $0$ \\
\hline
$6,7$ & $1$ & $0,1$ \\
\hline
$8,9$ & $1,2$ & $0,1$ \\
\hline
$\vdots$& $\vdots$ & $\vdots$\\
\hline
\end{tabular}
\caption{Values of $l$ which yield curves with $\tau_\beta>0$ for various values of $m$}
\label{lvalues}
\end{center}
\end{table}

Similarly, we have that
$\theta^-(\omega,l^-)\in\frac{\pi}{m}\left(\frac{4l^{+}+1}{2},2l^{+}\right)$.
Thus, for $\theta^-$ to be in the appropriate interval so that $\tau_\beta>0$,
we require that
\[ 2k\pi\le \frac{4l^++1}{2m}\pi < (2k+\frac{1}{2})\pi,\quad
\mbox{or}\quad
 2k\pi< \frac{4l^++2}{2m}\pi \le (2k+\frac{1}{2})\pi. \]
Analysis as above shows that these inequalities can only be satisfied
for $m\ge 2$ and $k=0$. This gives rise to the values of $l^-$ shown
in Table~\ref{lvalues}.
Unfortunately, the nonlinearity of the function $h(\omega)$ makes it
impossible to analytically compute the $\omega$ interval giving rise to the
curves described by Table~\ref{lvalues}. However, the curves can be computed
numerically.  Figure~\ref{bifurcation_curves1} shows an ensemble of some specific curves in the parameter space $(\beta,\tau_\beta)$ to give credence to our theoretical results. It is important to note that the only relevant quadrants are the first and the second, since $\tau_\beta>0$.
\begin{figure}[t!]
    \centering
    \begin{subfigure}[t]{0.5\textwidth}
        \centering
        \includegraphics[scale=.4]{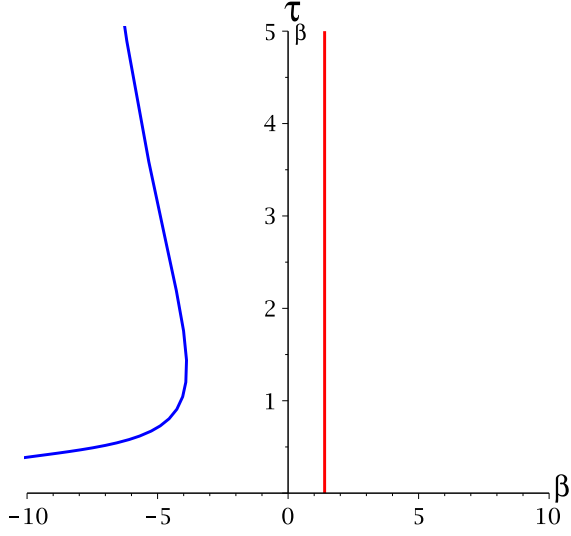}
        \caption{$m=3$.}
        \label{fig:bif_curves4}
    \end{subfigure}%
    ~
    \begin{subfigure}[t]{0.5\textwidth}
        \centering
        \includegraphics[scale=.4]{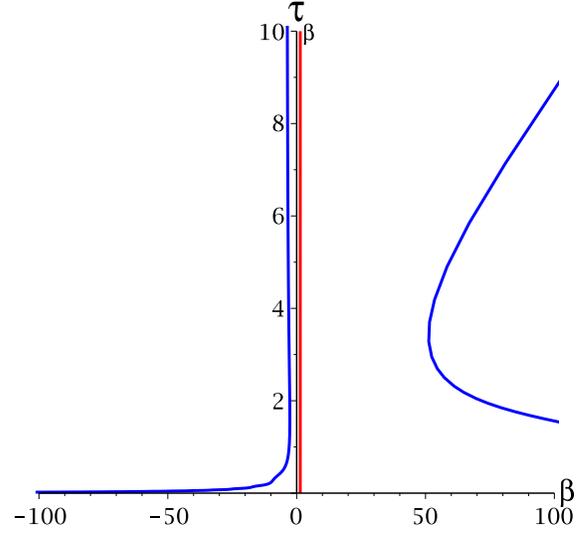}
        \caption{$m=5$.}
        \label{fig:bif_curves3}
    \end{subfigure}
    ~
    \begin{subfigure}[t]{0.5\textwidth}
        \centering
        \includegraphics[scale=.4]{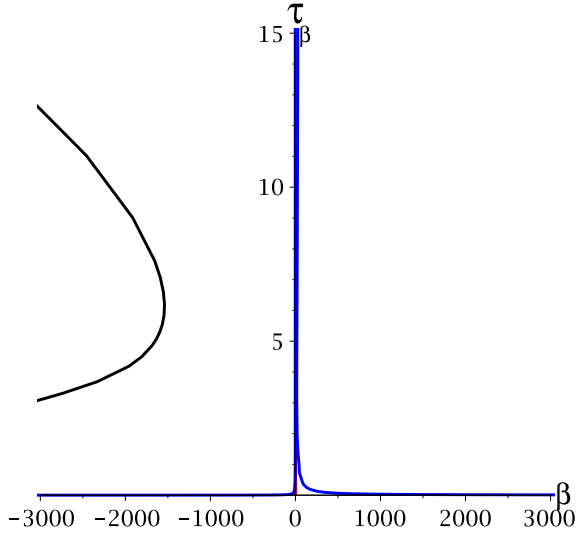}
        \caption{$m=7$.}
        \label{fig:bif_curves1}
    \end{subfigure}%
    ~
    \begin{subfigure}[t]{0.5\textwidth}
        \centering
        \includegraphics[scale=.4]{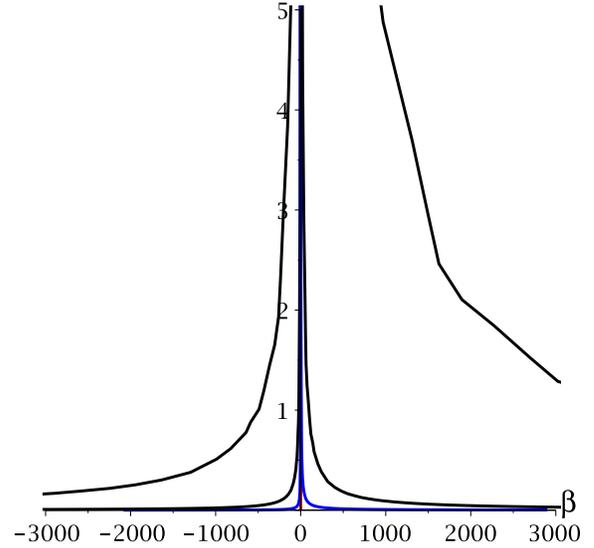}
        \caption{$m=30$.}
        \label{fig:bif_curves2}
    \end{subfigure}
    \caption{A display of some bifurcation curves (blue, black) in the $(\beta,\tau_\beta)$ parameter space, with $m$ as indicated.
The parameters used here are $\alpha=-0.4$, $\tau_\alpha=1$ and $l$ as
shown in Table~\ref{lvalues}. The range of $\omega>0$ was varied to accommodate different scales. Only six of 15 curves for $m=30$ appear at the chosen scale.
The red line is the line $\beta=1-\alpha$.
}
    \label{bifurcation_curves1}
\end{figure}

\subsection{Preservation of stability when traversing curves in parameter space}\label{number-of-roots-with-real-part}
\noindent
Starting from the characteristic equation \eqref{char_eqn},
we can divide through by $a^m$, substitute $\tau_\beta=\frac{m}{a}$, and rearrange
to obtain
\begin{equation}
\displaystyle
\left(\lambda+1-\alpha e^{-\lambda\tau_\alpha}\right)\left(\frac{\lambda\tau_\beta}{m}+1\right)^m - \beta  = 0.
\label{char_eqn_tau}
\end{equation}
Let us consider first the variation of $\beta$.  Differentiating equation
\eqref{char_eqn_tau} with respect to $\beta$, assuming $\lambda=\lambda(\beta)$, and
rearranging yields
\begin{equation}
\displaystyle
\frac{d\lambda}{d\beta} = \frac{1}{\left(1+\alpha \tau_\alpha e^{-\lambda\tau_\alpha}\right)\left(\frac{\lambda\tau_\beta}{m}+1\right)+\tau_\beta\left(\lambda+1-\alpha e^{-\lambda\tau_\alpha}\right)}\;.
\end{equation}
Substituting $\lambda=0$ then leads to
\begin{equation}
\displaystyle
\left.\frac{d\lambda}{d\beta}\right|_{\lambda=0} = \frac{1}{1+\alpha \tau_\alpha +\tau_\beta\left(1-\alpha \right)}\;.
\end{equation}
Thus, along the line $\beta=1-\alpha$, we have
\begin{equation}
\frac{d\lambda}{d\beta}
\left\{
\begin{array}{cc}
>0 & \mbox{ if } 0\le \alpha<1 \\
\gtrless 0& \mbox{ if } \alpha>1 \mbox{ and }  \tau_\beta\gtrless \frac{1+\alpha\tau_\alpha}{\alpha -1}\\
\gtrless 0& \mbox{ if } \alpha<0 \mbox{ and }  \tau_\beta\lessgtr\frac{1+\alpha\tau_\alpha}{\alpha -1}
\end{array}
\right.
\end{equation}
Differentiating \eqref{char_eqn_tau} with respect to $\tau_\beta$,
assuming $\lambda=\lambda(\tau_\beta)$, and rearranging gives
\begin{equation}
\displaystyle
\frac{d\lambda}{d\tau_\beta} = \frac{\lambda\left(\alpha e^{-\lambda\tau_\alpha}-s-1\right)}{\left(1+\alpha \tau_\alpha e^{-\lambda\tau_\alpha}\right)\left(\frac{\lambda\tau_\beta}{m}+1\right)+\tau_\beta\left(\lambda+1-\alpha e^{-\lambda\tau_\alpha}\right)}\;.
\end{equation}
Substituting $\lambda=i\omega$, $\omega>0$, gives the following.
\begin{equation}
\displaystyle
\left(\left. \frac{d\lambda}{d\tau_\beta}\right|_{\lambda=i\omega}\right) = \frac{i\omega\left(\alpha e^{-i\omega \tau_\alpha}-i\omega-1\right)}{\left(1+\alpha \tau_\alpha e^{-i\omega \tau_\alpha}\right)\left(\frac{i\omega\tau_\beta}{m}+1\right)+\tau_\beta\left(i\omega+1-\alpha e^{-i\omega \tau_\alpha}\right)}\;.
\end{equation}
This leads to
\begin{equation}
\displaystyle
\mbox{Re}\left(\left. \frac{d\lambda}{d\tau_\beta}\right|_{\lambda=i\omega}\right)=\frac{AC+BD}{C^2+D^2}\;,
\label{real-part}
\end{equation}
where
\begin{equation}
\displaystyle
\begin{array}{lcl}
A &=& \omega [\omega+\alpha\sin(\omega \tau_\alpha)]\;,\\
B &=& \omega[\alpha\cos(\omega \tau_\alpha)-1]\;,\\
C &=& 1+\alpha \tau_\alpha\cos(\omega \tau_\alpha)+\frac{\alpha\omega\tau_\beta\tau_\alpha}{m}\sin(\omega \tau_\alpha)+\tau_\beta\left[1-\alpha\cos(\omega \tau_\alpha) \right]\\
  &=& \frac{\alpha}{a}(a\tau_\alpha-m)\cos(\omega \tau_\alpha)+\frac{\alpha\omega \tau_\alpha}{a}\sin(\omega \tau_\alpha)+\frac{1}{a}(m+a)\;,\\
D &=& -\alpha \tau_\alpha\sin(\omega \tau_\alpha)+\frac{\omega\tau_\beta}{m}\left[1+\alpha \tau_\alpha\cos(\omega \tau_\alpha)\right]+\tau_\beta\left[\omega +\alpha\sin(\omega \tau_\alpha)\right]\\
  &=& \frac{\alpha}{a}(m-a \tau_\alpha)\sin(\omega \tau_\alpha)+\frac{\alpha\omega \tau_\alpha}{a}\cos(\omega \tau_\alpha)+\frac{\omega}{a}(1+m)\;.
\end{array}
\end{equation}
We note that the numerator of the right hand side of \eqref{real-part} is given by
\begin{equation}
\displaystyle
\begin{array}{lcl}
AC+BD &=& \alpha\omega \tau_\alpha\left[\sin(\omega \tau_\alpha)+\omega\cos(\omega \tau_\alpha)\right]+\frac{\alpha\omega^2}{a}(1-\tau_\alpha)\cos(\omega \tau_\alpha)\\
& & +\frac{\alpha\omega}{a} \left(\omega^2 \tau_\alpha+a\right)\sin(\omega \tau_\alpha)
 +\frac{\omega^2}{a}\left(a-1+\alpha^2 \tau_\alpha\right)\;.
\end{array}
\label{numerator: real part}
\end{equation}
We will show how $\mbox{Re}\left(\left.\frac{d\lambda}{d\tau_\beta}\right|_{\lambda=i\omega}\right)$ is related to $\frac{d\beta^{\pm}}{d\omega}$.
Differentiation of \eqref{beta_expr} with respect to $\omega$ leads to
\begin{equation}
\frac{d\beta^{\pm}}{d\omega}
= \beta^{\pm} \left[ \frac{\hat A \hat A'+\hat B \hat B'}{\hat{A}^2+\hat{B}^2} + m\tan(\theta)\frac{d\theta}{d\omega} \right]\;,
\label{dbdw}
\end{equation}
where $\hat A :=A/\omega$ and $\hat B := B/\omega$. Using equations \eqref{tantheta} and \eqref{tanmtheta}, and simplifying, we arrive at
\begin{equation}
\begin{array}{lcl}
\dss
\frac{d\beta^\pm}{d\omega} &=&
\dss
\left.\frac{\beta^\pm}{\hat{A}^2+\hat{B}^2}
\right[\omega +\alpha \sin(\omega \tau_\alpha) +\alpha \tau_\alpha \left[\sin(\omega \tau_\alpha)+\omega\cos(\omega \tau_\alpha)\right] \\
&&\displaystyle \left. \hspace*{.75in}+\, \frac{\omega}{a} \left( \alpha (1-\tau_\alpha)\cos(\omega \tau_\alpha) + \alpha \tau_\alpha \omega \sin (\omega \tau_\alpha) +\alpha^2\tau_\alpha-1 \right)\right]\\
&=&\dss \omega\beta^\pm\,\left(\frac{AC+BD}{A^2+B^2}\right)\;.
\end{array}
\label{dbeta-domega-expression}
\end{equation}
By substituting \eqref{dbeta-domega-expression} into \eqref{real-part}, we arrive at the expression
\begin{equation}
\displaystyle
\mbox{Re}\left(\left. \frac{d\lambda}{d\tau_\beta}\right|_{\lambda=i\omega}\right)
=\frac{1}{\omega\beta^\pm}\left(\frac{A^2+B^2}{C^2+D^2}\right)\frac{d\beta^\pm}{d\omega}\;.
\label{relationship-between-derivatives}
\end{equation}
From the relationship \eqref{relationship-between-derivatives}, $\beta^{+}>0$, $\beta^{-}<0$, and
$\omega>0$, it is clear that
\begin{equation}
\mbox{Re}\left(\left. \frac{d\lambda}{d\tau_\beta}\right|_{\lambda=i\omega}\right)\glt 0
\ \Longleftrightarrow \ \frac{d\beta^{+}}{d\omega}\glt 0 \ \mbox{ or }\ \frac{d\beta^{-}}{d\omega}\lgt 0\;.
\label{derivative-sign-conditions}
\end{equation}
This leads to the following proposition.
\begin{prop}\label{directionprop}
The number of roots of the characteristic equation \eqref{char_eqn} with
positive real parts is increasing/decreasing when $\beta^+(\omega)$ is an
increasing/decreasing function of $\omega$ and the opposite for
$\beta^-(\omega)$.
\end{prop}
\section{Exact Stability Region}
We are now in a position to more fully describe the region in the
$(\beta,\tau_\beta)$ parameter space where the trivial solution of 
\eqref{rescaled_eqn}
is asymptotically stable. We will do so by putting together the results of Theorems~\ref{instabthm}-\ref{rouche}
and Proposition~\ref{directionprop} and the description
of the curves where the characteristic equation has a pair of pure imaginary roots.

\subsection{The case $m=1.$}
When $m=1$ the analysis simplifies considerably. In this case there can
be at most one curve along which the characteristic equation has a pair of
pure imaginary roots. We describe this curve and how/when it forms part of
the boundary of the stability region below.

To begin, we set $m=1$ and $\tau_\beta=1/a$ in equations \eqref{tantheta},
\eqref{tanmtheta} and \eqref{beta} to obtain
\begin{equation}
\tau_\beta= \frac{h(\omega)}{\omega}
\label{taum1}
\end{equation}
\begin{equation}
\beta=\frac{(1-\alpha\cos(\omega\tau_\alpha))^2+(\omega+\alpha\sin(\omega\tau_\alpha)^2}{1-\alpha\cos(\omega\tau_\alpha)}
= (1-\alpha\cos(\omega\tau_\alpha))+\frac{(\omega+\alpha\sin(\omega\tau_\alpha)^2}{1-\alpha\cos(\omega\tau_\alpha)}
\label{beta_m1}
\end{equation}
where $h(\omega)$ is given by \eqref{tanmtheta}.
For fixed $\alpha$ and $\tau_\alpha$, these equations define a curve in the $\beta,\tau_\beta$ parametrically in
terms of $\omega$.
Further, note the following limits.
\begin{equation}
\dss \lim_{\omega\rightarrow 0} \tau_\beta(\omega) = \lim_{\omega\rightarrow 0}
\frac{-(1+\alpha\tau_\alpha\frac{\sin(\omega\tau_\alpha)}{\omega\tau_\alpha})}{1-\alpha\cos(\omega\tau_\alpha)}
=-\frac{1+\alpha\tau_\alpha}{1-\alpha},\label{taulimit}
\end{equation}
\[
\dss\lim_{\omega\rightarrow 0} \beta=1-\alpha,
\]
and
\[ \tau_\beta=0 \Rightarrow h(\omega)=0 \Rightarrow \omega+\alpha \sin(\omega\tau_\alpha)=0
\Rightarrow \beta=1-\alpha\cos(\omega\tau_\alpha) \]
This leads to the following.

\begin{prop}\label{m1P1}
If $m=1$, $0<\alpha<1$ and $\tau_\alpha\ge 0$,
then the trivial solution is
asymptotically stable in the region $\beta<1-\alpha$, $\tau\ge 0.$
\end{prop}

\begin{proof}
In this case, from equation \eqref{beta_m1} $\beta\ge 1-\alpha>0$.
Thus the curve $(\beta,\tau_\beta)$ defined by \eqref{taum1}-\eqref{beta_m1}
can only lie in
Quadrants I or IV. Since in Quadrant I the curve lies to the right of the
curve $\beta=1-\alpha$, by Theorem~\ref{instabthm}
it won't affect the stability and the result follows.
\end{proof}

\begin{prop}\label{m1P2}
If $m=1,$ $-1<\alpha<0$  and $\tau_\alpha\le-1/\alpha$, then
the trivial solution is
asymptotically stable in the region $\beta<1-\alpha$, $\tau\ge 0$.
\end{prop}

\begin{proof}
In this case, from equation \eqref{beta_m1} $\beta>1+\alpha>0$.  So the curve
can only lie in Quadrant I or IV.  By Lemma~\ref{Ssignlem} $h(\omega)<0$ for
all $\omega>0$, which implies that $\tau_\beta<0$ for all $\omega>0$, i.e.,
the curve lies in Quadrant IV. Hence it does not affect the stability.
The result then follows from Theorems~\ref{instabthm} and \ref{rouche}.
\end{proof}
The results of Propositions~\ref{m1P1}-\ref{m1P2} are shown in Figure
\ref{m1stabfig}(a).
\begin{figure}[t!]
\centering
\begin{subfigure}[t]{0.49\textwidth}
\centering
\caption{$0\le \alpha<1,$ $\tau_\alpha\ge 0\ $ or
$\ -1<\alpha<0,\ \tau_\alpha\le -1/\alpha$}
\includegraphics[scale=.4]{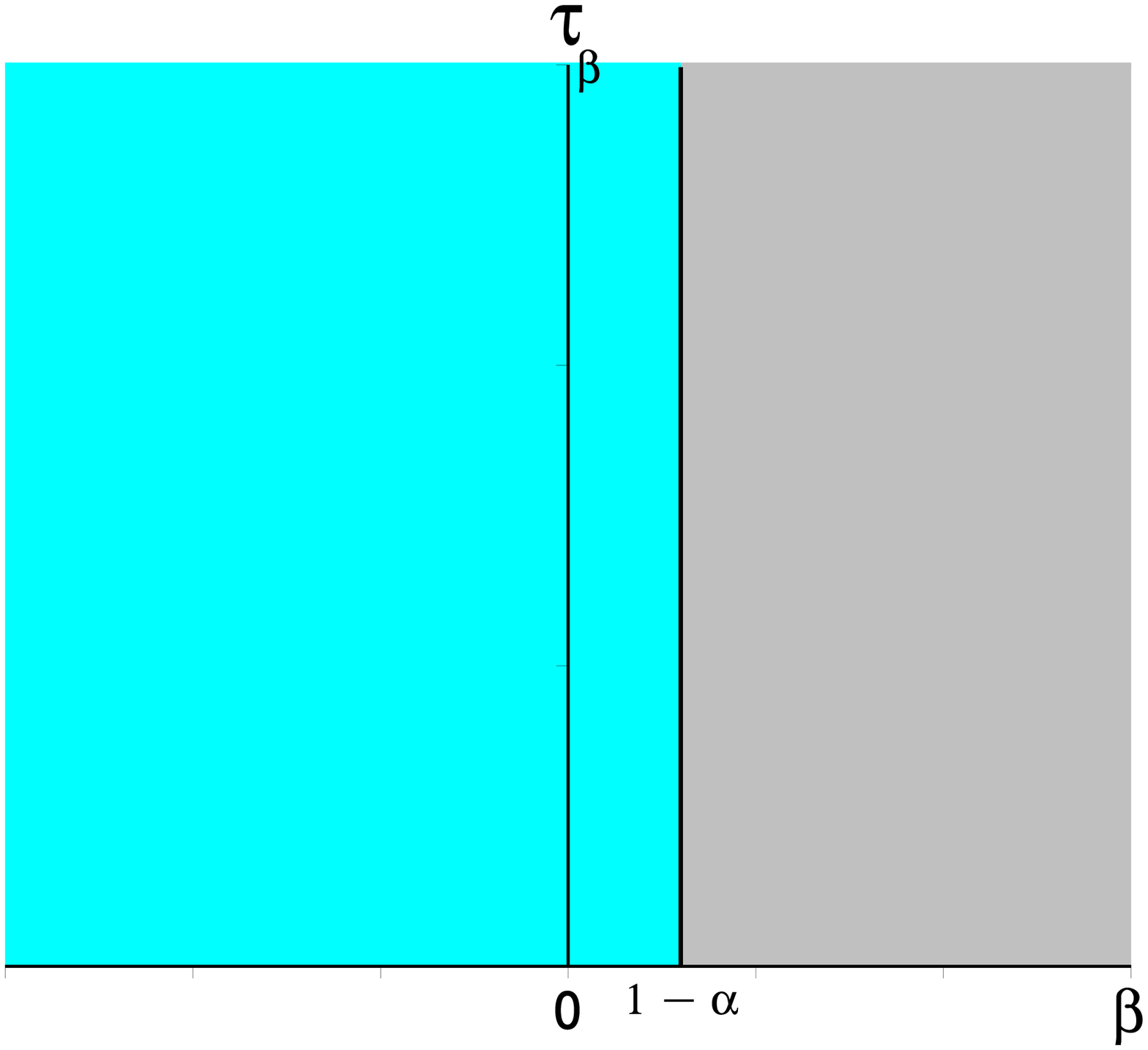}
\end{subfigure}~
\begin{subfigure}[t]{0.49\textwidth}
\centering
\caption{$-1<\alpha<0,\ \tau_\alpha>-1/\alpha$}
\includegraphics[scale=.4]{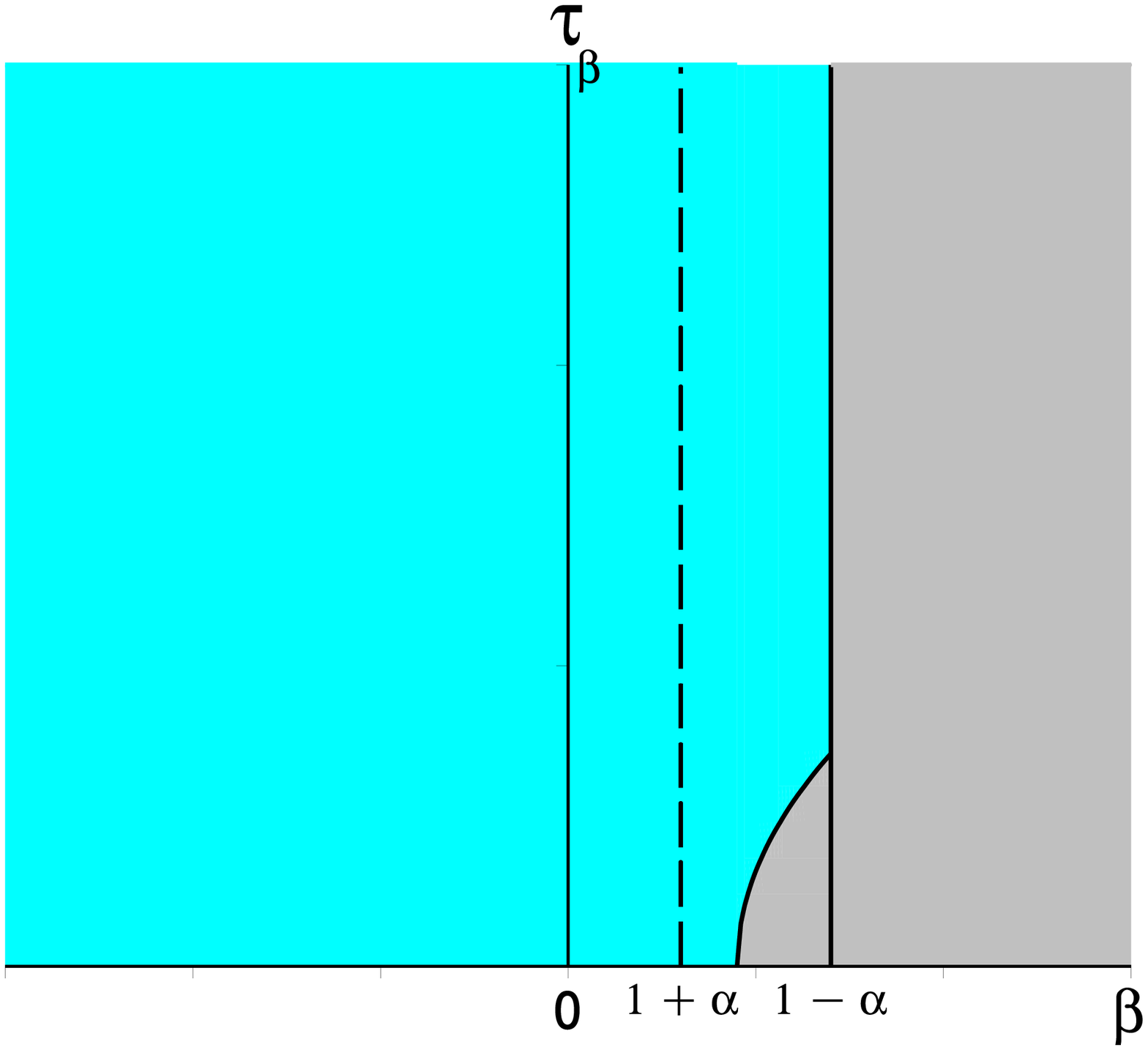}
\end{subfigure}
\caption{Exact stability regions for $m=1$ and $|\alpha|<1$.
Regions shaded in grey correspond to instability.
Regions shaded in cyan correspond to asymptotic stability}
\label{m1stabfig}
\end{figure}

Proposition~\ref{m1P2} requires $|\alpha\tau_\alpha|$ small enough for stability.
If $\tau_\alpha>-1/\alpha$ then $\tau_\beta>0$ as $\omega\rightarrow 0$ and
part of the curve lies in Quadrant I starting on the line $\beta=1-\alpha$ and
ending the right of the line $\beta=1+\alpha$.  Consideration of
the results of the previous section shows that the stability region lies
outside this curve.  See Figure~\ref{m1stabfig}(b).

When $|\alpha|>1$ the curve has discontinuities when
$1-\alpha\cos(\omega\tau_\alpha)=0$ and the stability boundary becomes
more complex. However, we can say the following.
\begin{prop}\label{m1P3}
If $m=1$, $\alpha>1$ and $\tau_\alpha<u^*/\alpha$, then the trivial solution
is unstable above the curve defined by
equations \eqref{taum1}--\eqref{beta_m1} for $ 0\le\omega\le\omega^*$
where $u^*$ is defined in Lemma~\ref{Ssignlem} and $\omega^*$ is
the smallest positive zero of $1-\alpha\cos(\omega\tau_\alpha)$.
\end{prop}
\begin{proof}
Let $S(\omega)$ be as defined in the proof of Lemma~\ref{Ssignlem}
and $C(\omega)=1-\alpha\cos(\omega\tau_\alpha)$.
Since $\tau_\alpha<u^*/\alpha$, $S(\omega)>0$ for $\omega>0$ thus the
curve defined by equations \eqref{taum1}--\eqref{beta_m1} will either lie
in Quadrant IV, in which case
it does not affect the stability, or in Quadrant II. The branch of the curve for
$\omega\in [0,\omega^*$, $C(\omega))<0$ lies in Quadrant II
and emanates from the line $\beta=1-\alpha$ when $\omega=0$.
As $\omega\rightarrow \omega^*$, $\tau_\beta\rightarrow\infty$
and $\beta\rightarrow -\infty$.  Using the results of the previous section
shows that stability is lost along this branch of the curve.  It can be shown
that all other branches of the curve in Quadrant II lie below this curve.
The result follows.
\end{proof}
\begin{prop}\label{m1P4}
If $m=1$, $\alpha<-1$ and $\tau_\alpha\le -1/(2\alpha)$, then is the trivial
solution is asymptotically stable for $1+\alpha\le\beta< 1-\alpha$.
\end{prop}
\begin{proof}
Let $S(\omega)$ be as defined in the proof of Lemma~\ref{Ssignlem}
and $C(\omega)=1-\alpha\cos(\omega\tau_\alpha)$. Since $\tau_\alpha\le -1/\alpha$,
 $S(\omega)>0$ for $\omega>0$ thus the
curve defined by equations \eqref{taum1}--\eqref{beta_m1} will either lie
in Quadrant IV, in which case it does not affect the stability, or in Quadrant II.
Consideration of Theorems~\ref{instabthm} and ~\ref{rouche2} gives the result.
\end{proof}
The results of these Propositions are shown in Figure~\ref{m1stabfig2}.
In the white regions of this Figure, the curve defined by \eqref{taum1}--\eqref{beta_m1} will have an infinite number of branches defined by the $\omega$ values
where $1-\alpha\cos(\omega\tau_\alpha)>0$.  Inside these branches
trivial solution will be unstable.
\begin{figure}[t!]
\centering
\begin{subfigure}[t]{0.5\textwidth}
\centering
\caption{$\alpha>1$, $ \tau_\alpha\le u^*/\alpha$}
\includegraphics[scale=.4]{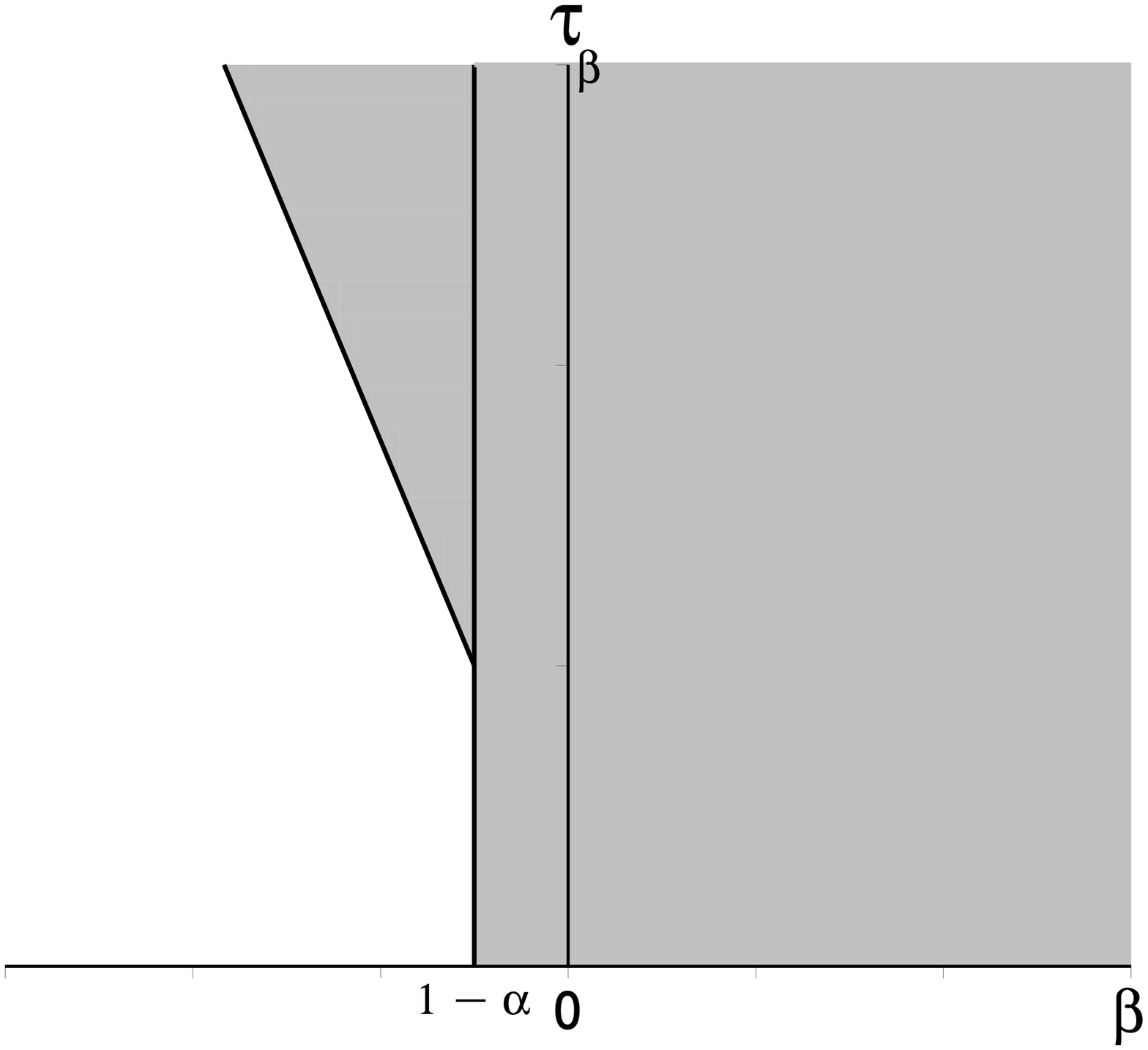}
\end{subfigure}~
\begin{subfigure}[t]{0.5\textwidth}
\centering
\caption{$\alpha<-1,\ \tau_\alpha\le -1/(2\alpha)$}
\includegraphics[scale=.4]{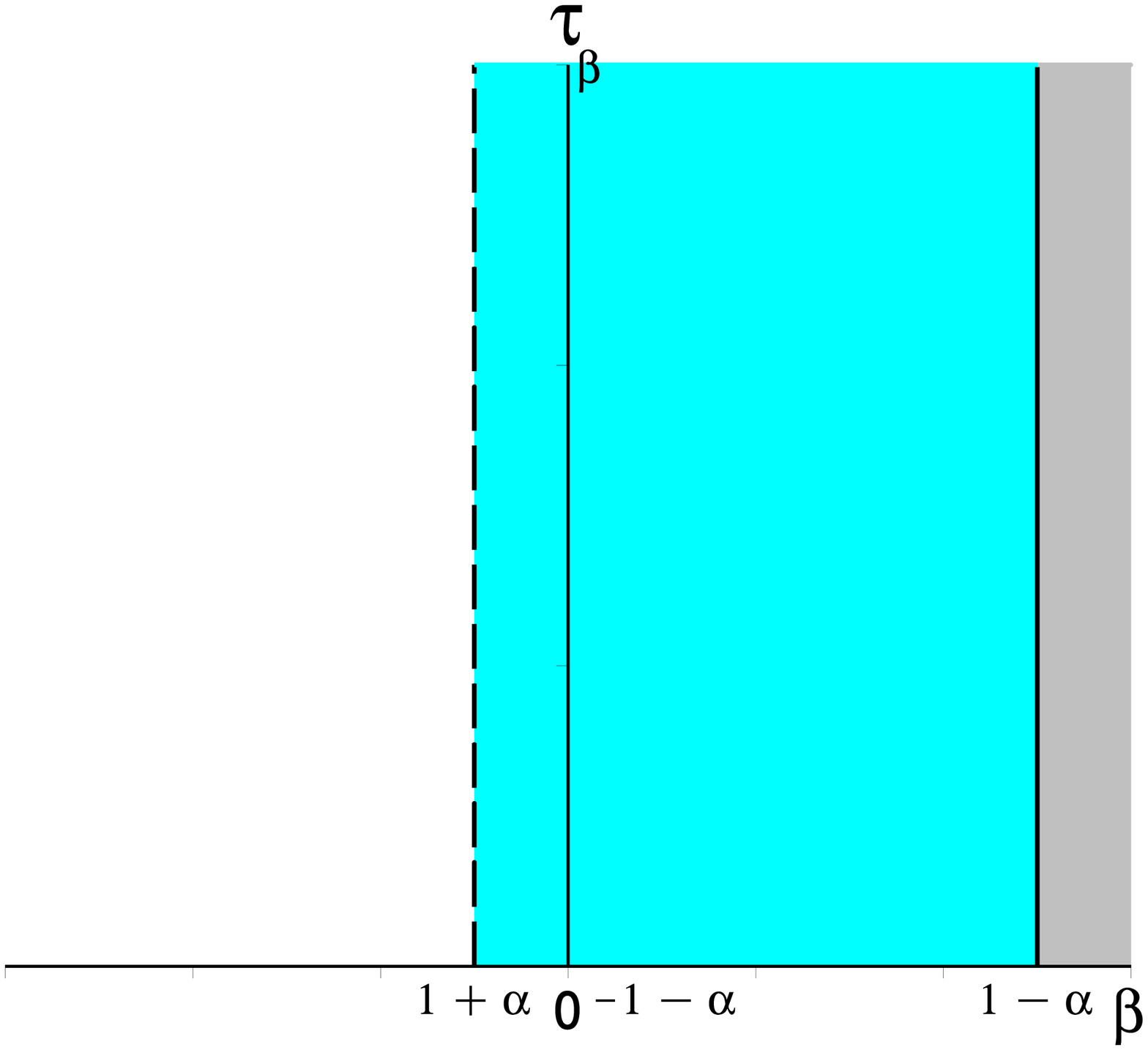}
\end{subfigure}
\caption{Stability regions for $m=1$ and $|\alpha|>1$.
Regions shaded in grey correspond to instability. Regions shaded
in cyan correspond to asymptotic stability. In the white regions there
may be stability or instability depending on the parameter values.
See text.}
\label{m1stabfig2}
\end{figure}
\subsection{The case $m>1.$}
Our first result shows how Proposition~\ref{m1P1} generalises to
higher values of $m$.
\begin{prop}\label{m345prop}
If $\,0<\alpha<1$ then the right boundary of the region of asymptotic
stability of the trivial solution is the line $\beta=1-\alpha$.
If, in addition, $m=2,3,4,5$, then the left boundary is
the curve $(\beta^-(\omega,0),\tau_\beta^-(\omega,0)$.
\end{prop}
\begin{proof}
Recall that Theorem~\ref{instabthm} established that the trivial solution is
unstable in the region $\beta>1-\alpha$ and that Theorem~\ref{rouche} showed
that the trivial solution is asymptotically stable for $|\beta|<1-|\alpha|$,
$\tau_\beta\ge 0$. This stability will be maintained until parameter values
where the characteristic equation has a root with zero real part, which occurs
along the line $\beta=1-\alpha$ or along the curves defined by
\eqref{positive_cos} or \eqref{negative_cos}.

Since  $\,0<\alpha<1$, it follows from Theorem~\ref{rouche} that the
trivial solution is asymptotically stable for $|\beta|<1-\alpha$,
which gives the first result. Note that this is also confirmed by the fact
that, for $0<\alpha<1$, the curves \eqref{positive_cos} satisfy
\[\beta^+(\omega,l^+)\ge (1-\alpha)\]
for all $l^+$ and $m=1,2,\ldots$. That is, these curves
lie to the right of the line $\beta=1-\alpha$.

For $m=2,3,4,5$, from Table~\ref{lvalues} there is one curve of
\eqref{negative_cos} lying in the second quadrant,
 $(\beta^-(\omega,0),\tau_\beta^-(\omega,0)$.  Thus this curve must
form the left boundary of the stability region.
\end{proof}
The results of Proposition~\ref{m345prop} are illustrated in
Figure~\ref{m345fig}.
\begin{figure}
\begin{center}
 \includegraphics[scale=.4]{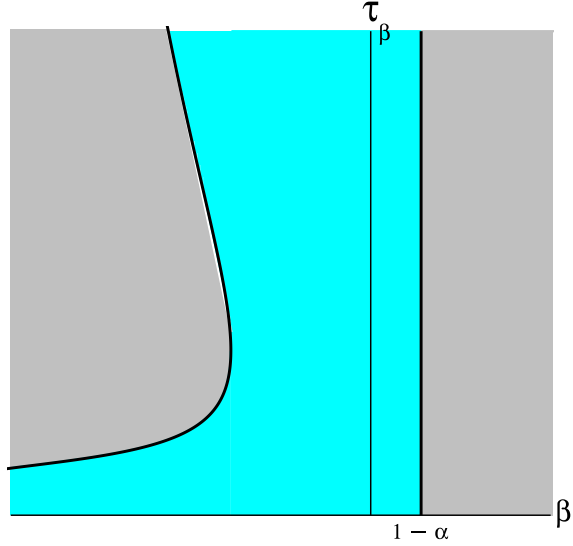}\\
\end{center}
\caption{Stability region for $m=2,3,4,5$ with
 $0<\alpha<1$.}
\label{m345fig}
\end{figure}
Once one has determined the curves where the characteristic equation
has pure imaginary eigenvalues, the stability region can be determined
by using Theorems~\ref{instabthm}-\ref{rouche2} and Proposition~\ref{directionprop}.
It is possible to make general statements about the stability region under
conditions other than those of Proposition~\ref{m345prop}, but such statements
become more complicated as the index $m$ increases.  Instead, we will make some observations
about various cases and illustrate them with examples.

Proposition~\ref{m345prop} may be generalised to higher values of $m$, but the
left boundary of the stability region may consist of multiple curves; see Figure~\ref{Otherparamfig}(a) for an example.
It appears that Proposition~\ref{m345prop} has an analogue for
$-1<\alpha<0$ and $\tau_\alpha$ sufficiently small. See Figure~\ref{bifurcation_curves1},
where in all cases, the stability region is schematically the same as depicted in Figure~\ref{m345fig}.
However, for large enough $m$ or $\tau_\alpha$, this is no longer true; see
Figure~\ref{Otherparamfig}(b) for an example.

Our results for $m>1$ are all for $|\alpha|\le 1$. In the case that $|\alpha|>1$,
the expressions for $\beta^\pm$ and $\tau_\beta^\pm$ are as given in
\eqref{positive_cos}-\eqref{negative_cos}. However, the expressions for $\theta^{\pm}$ are modified to
\[
\theta^{\pm}(\omega,l^{\pm})=
\left\{
\begin{array}{lcl}
\frac{1}{m}\Arctan[h(\omega)]+2l\pi\;&,&\;\mbox{if}\;\; \pm(1-\alpha\cos\tau_\alpha) \ge 0\;, \\
\frac{1}{m} \Arctan[h(\omega)]+(2l+1)\pi\;&,&\;\mbox{if}\;\; \pm(1-\alpha\cos\tau_\alpha) < 0\;.
\end{array}
\right.
\]
Using these expressions we can explore the stability region further.
It does appear that Proposition~\ref{m1P3} may be generalised to higher values of $m$,
but possibly with a more complicated left boundary of the stability region. This scenario is exemplified in
Figure~\ref{Otherparamfig}(c). One can use the fact \cite{Yuan_Campbell04} that when $\beta=0$,
the trivial solution is asymptotically stable if $\alpha<-1$ and
\[\tau_\alpha<\frac{1}{\sqrt{\alpha^2-1}}\arccos\left(\frac{1}{\alpha}\right)=\tau_{crit}. \]
In this situation, the boundary of the stability region will be formed by the parts of the curves 
of pure imaginary eigenvalues and the line $\beta=1-\alpha$ closest to the
$\beta$ axis. An example of this is shown in Figure~\ref{Otherparamfig}(d), where $-1/\alpha<\tau_\alpha<\tau_{crit}$.
\begin{figure}[t!]
\centering
\begin{subfigure}[t]{0.5\textwidth}
\centering
\caption{$m=100,\ \alpha=0.95,\ \tau_\alpha=4.6$}
 \includegraphics[scale=.4]{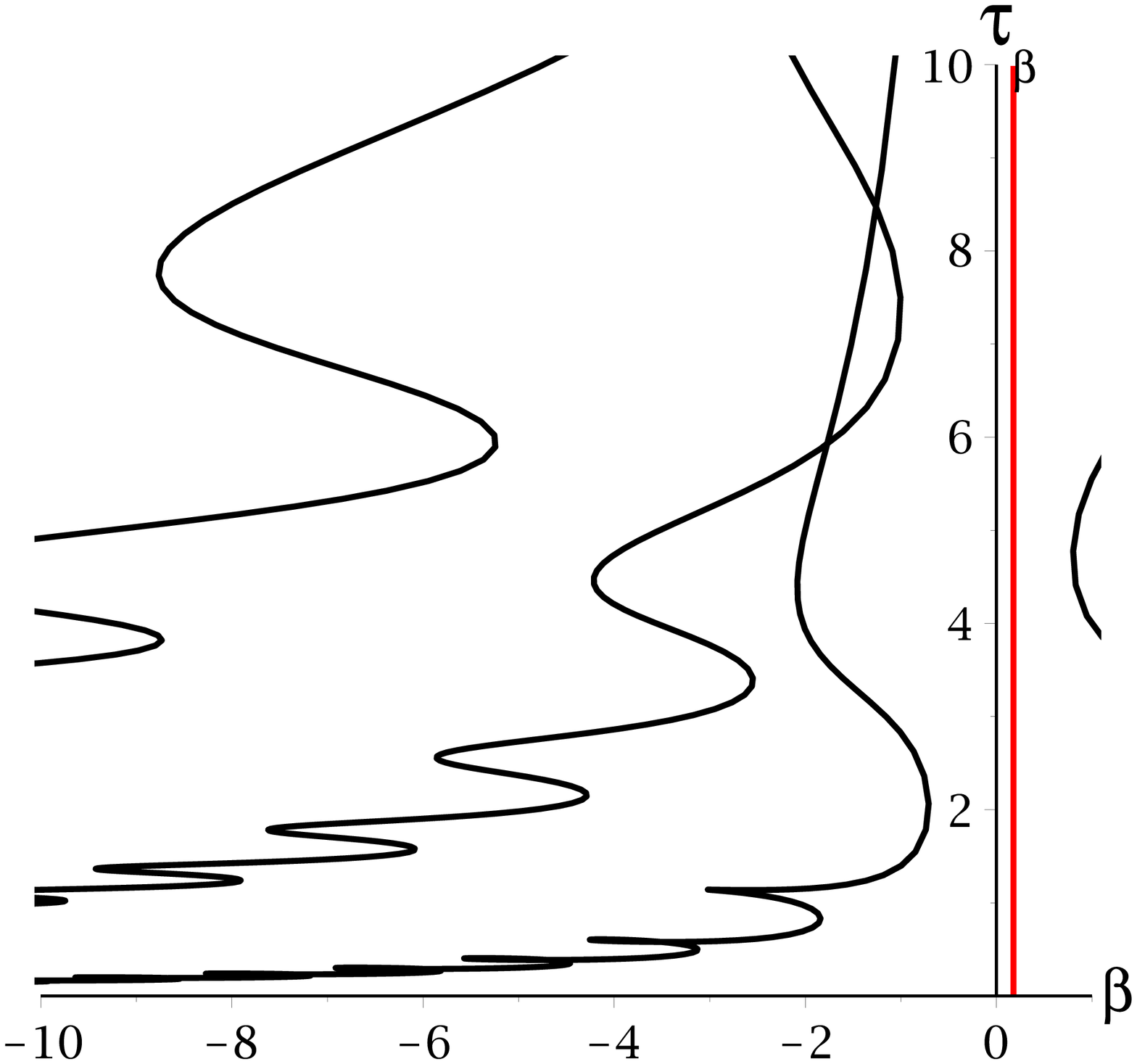}
\end{subfigure}~
\begin{subfigure}[t]{0.5\textwidth}
\centering
\caption{$m=100,\ \alpha=-0.95,\ \tau_\alpha=1$}
 \includegraphics[scale=.4]{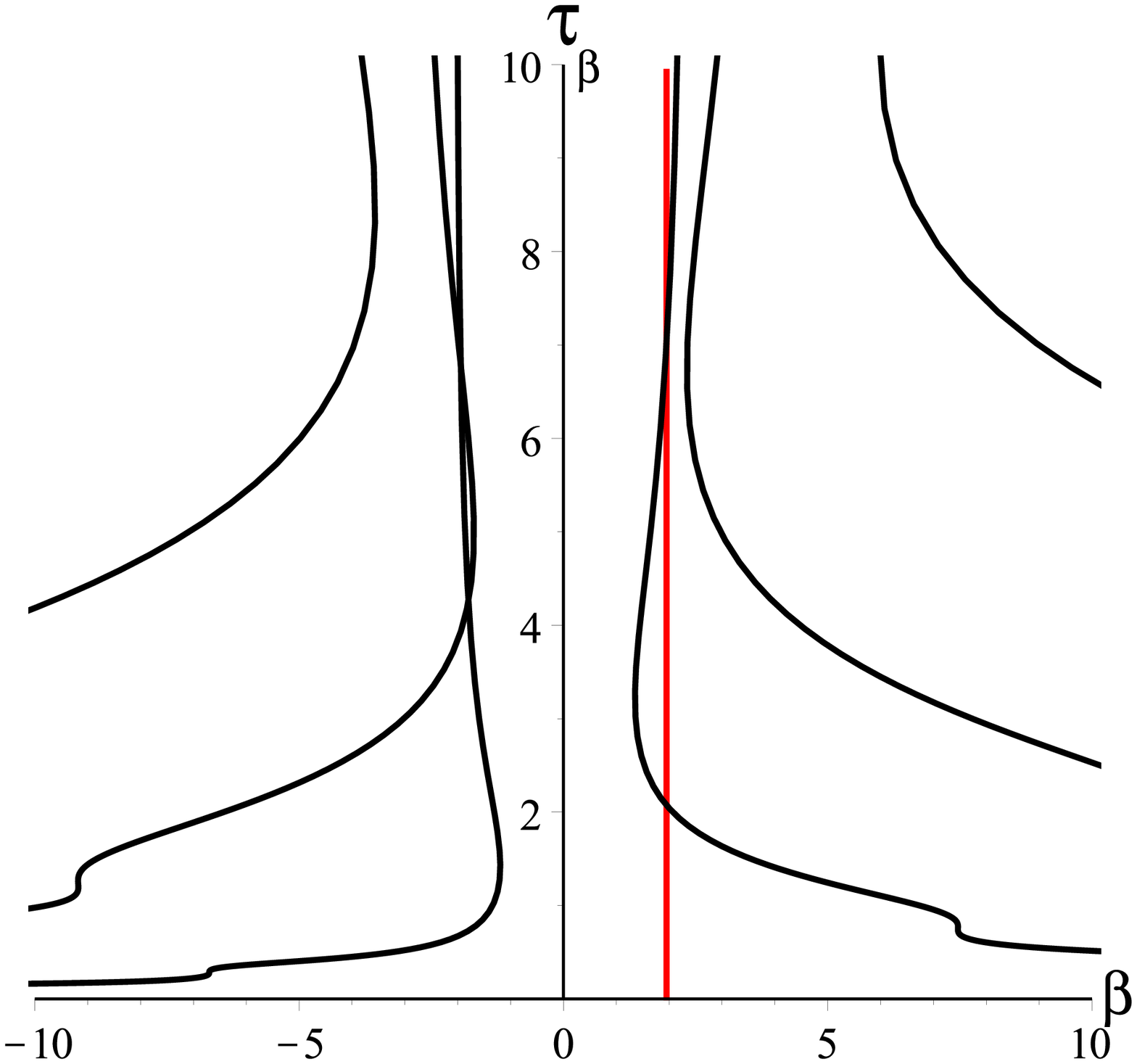}
\end{subfigure}
\begin{subfigure}[t]{0.5\textwidth}
\centering
\caption{$m=7,\ \alpha=1.5,\ \tau_\alpha=3$}
 \includegraphics[scale=.4]{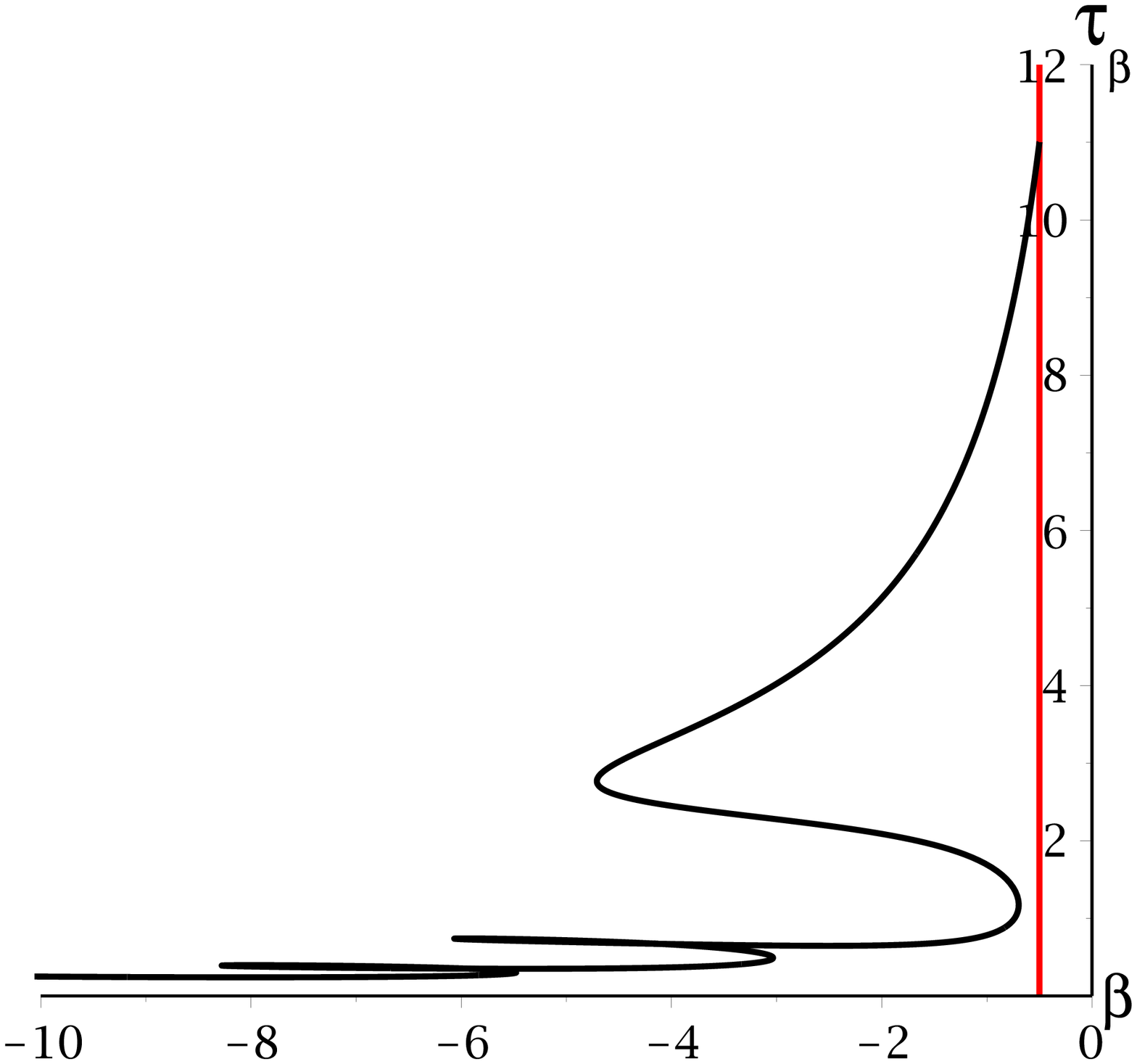}
\end{subfigure}~
\begin{subfigure}[t]{0.5\textwidth}
\centering
\caption{$m=7,\ \alpha=-1.2,\ \tau_\alpha=3$}
 \includegraphics[scale=.4]{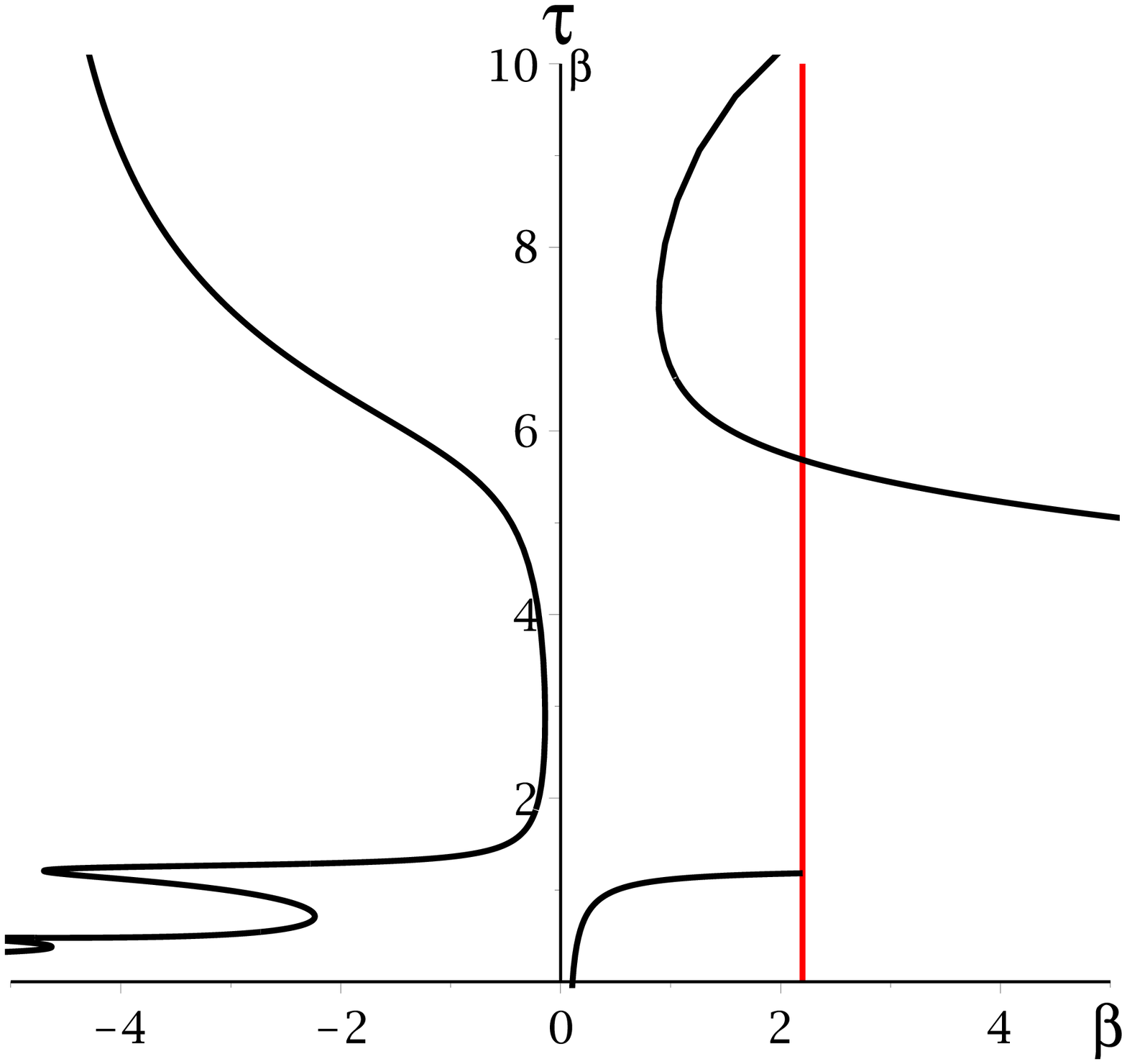}
\end{subfigure}
\caption{Region of stability of the trivial solution for other parameter values.
The boundary of the stability region is made up of the parts of the black curves
and the red line which are
closest to the $\tau_\beta$ axis. In (a) and (b) different
curves correspond to different values of $l$ in the curve equations.}
\label{Otherparamfig}
\end{figure}

\section{Conclusions}
We have studied the stability of the trivial solution of a linear, scalar delay
differential equation with two distributed time delays. We first gave {\em distribution-independent}
conditions for stability and instability. These conditions are identical to the
delay-independent conditions given for a scalar equation with two discrete time delays
\cite{Hale_Huang_1,Mah_etal_1}, and so our work generalises this result.

We then considered the case in which one time delay was discrete and the other was gamma distributed. We determined the stability
region in the parameter space of the strength of the gamma distributed term
and the mean of the gamma distributed time delay and considered how this region evolves as
the parameters of the discrete time delay term are varied.  Our results show that as
the parameter $m$ in the gamma distribution gets larger, the stability region
shrinks and the boundary of the stability region becomes more complicated.
More detail is given below. Our work generalises that of \cite{incube2013a}
which only considered the case $m=1$.


It is important to note that the parameter $m$ in the gamma distribution is a measure of the
variance of the distribution, with smaller $m$ yielding a distribution more tightly clustered around the mean.
In fact, if the mean of the distribution, $\tau_\alpha$, is held fixed, then in the
limit as $m\rightarrow\infty$ the gamma distribution approaches a Dirac distribution. In other words,
the time delay becomes discrete \cite{MacD_1}; see Figure~\ref{comparefig}(a). It has been
shown for systems with a single gamma distributed time delay that some results for the distributed
time delay approach those of the discrete time delay in the limit as $m\rightarrow\infty$
\cite{Cook_Gross,MacD_1,WXW,WXR}. We observe the same phenomenon in our model with
one discrete time delay and one gamma distributed time delay. Consideration of the
following limits
\[
\lim_{m\rightarrow\infty} m\tan\left(\frac{u}{m}\right)
=u, \qquad
\lim_{m\rightarrow\infty} \sec\left(\frac{u}{m}\right)^m
=1
\]
shows that as $m\rightarrow\infty$, the equations for the curves where the
characteristic equation \eqref{char_eqn} has a pair of pure imaginary roots reduce to
those for the case of two discrete time delays. These expressions are given by \cite{Yuan_Campbell04}:
\[ \beta^{\pm}(\omega) = \pm
\sqrt{[1-\alpha\cos(\omega \tau_\alpha)]^2 + [\omega+\alpha\sin(\omega \tau_\alpha)]^2},
\]
\[
\tau_\beta^\pm(\omega,l)= \left\{
\begin{array}{lcl}
\frac{1}{\omega}\left(\Arctan[h(\omega)]+2l\pi\right)\;&,&\;\mbox{if}\;\; \pm(1-\alpha\cos(\omega\tau_\alpha)\ge 0\;, \\
\frac{1}{\omega}\left(\Arctan[h(\omega)]+(2l+1)\pi\right)\;&,&\;\mbox{if}\;\; \pm(1-\alpha\cos(\omega\tau_\alpha)< 0\;.
\end{array}
\right.
\]
As can be seen in Figure~\ref{comparefig} (b), the stability region with
one distributed time delay and one discrete delay approaches that of a model with two discrete time delays in the limit
as $m\rightarrow\infty$.
\begin{figure}
\begin{center}
\begin{subfigure}[t]{0.5\textwidth}
\centering
\caption{\hspace{2in}\mbox{ }}
 \includegraphics[scale=.4]{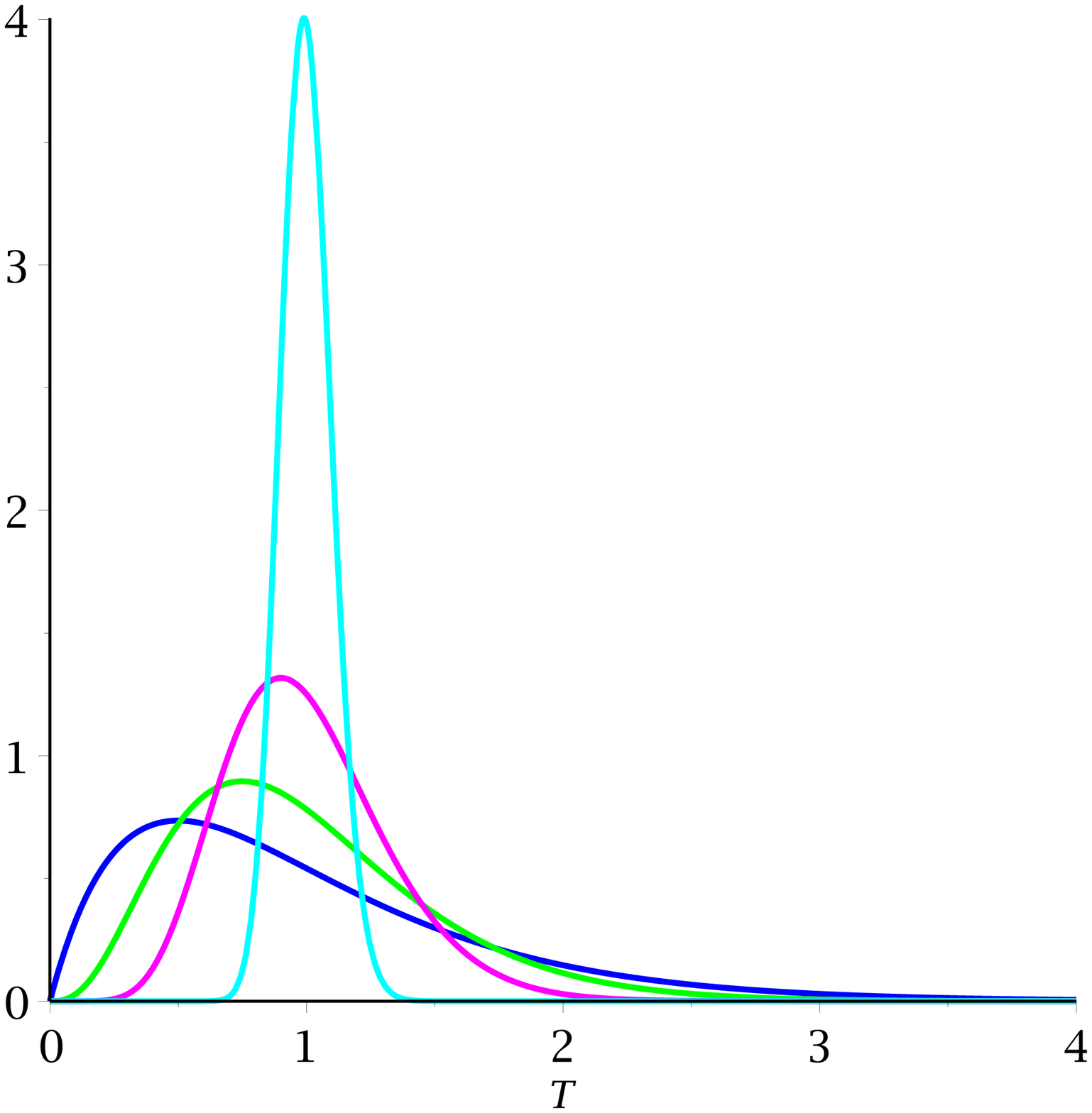}
\end{subfigure}~
\begin{subfigure}[t]{0.5\textwidth}
\centering
\caption{\hspace{2in}\mbox{ }}
 \includegraphics[scale=.4]{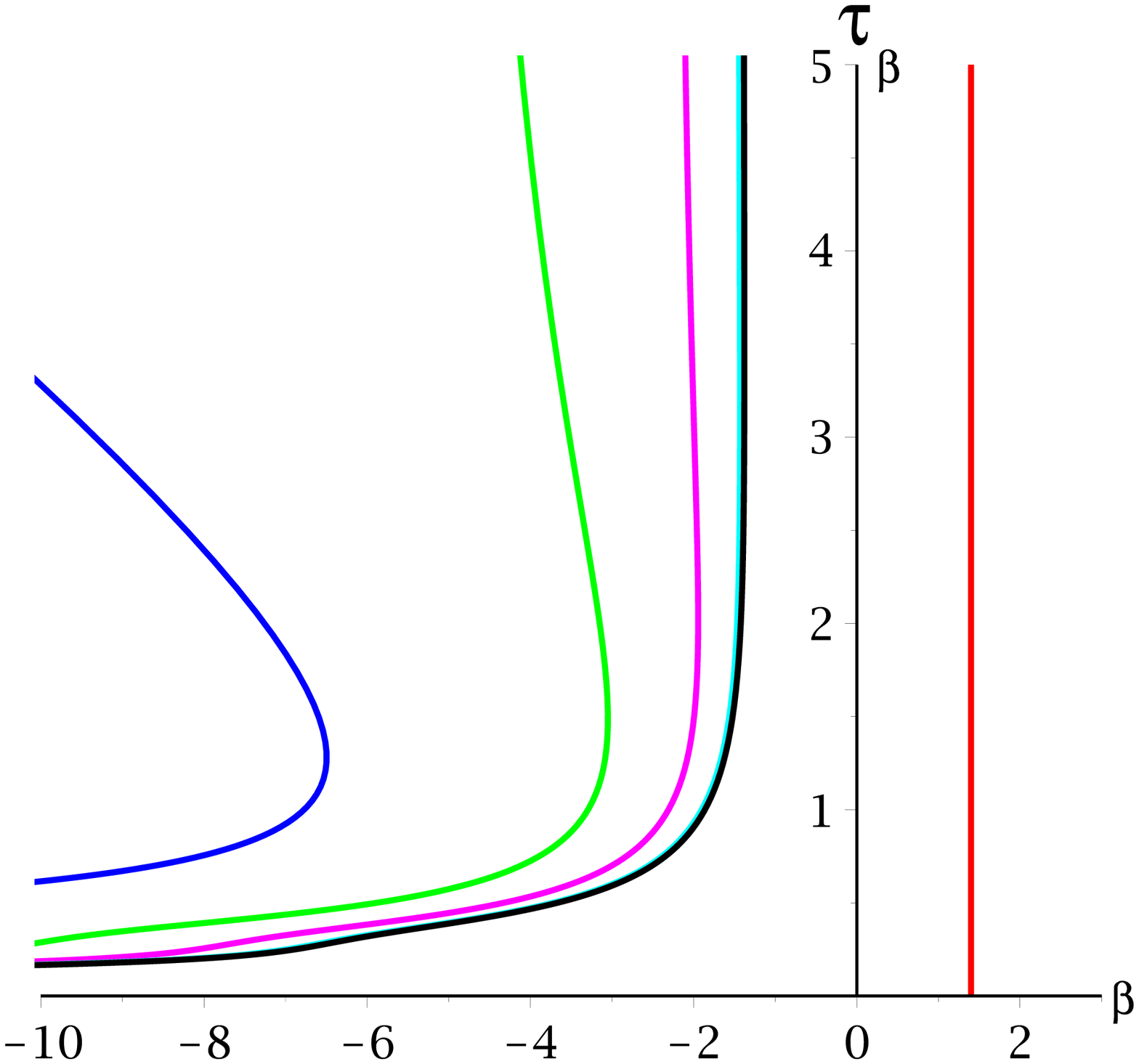}
\end{subfigure}
\end{center}
\caption{(a) Gamma distributions with $\tau_\alpha=1$ and $m=2,4,10,100$.
(b) Comparison of the stability region for the model with
one discrete and one gamma distributed delay
$m=2,4,10,100$ (blue,green,magenta,cyan) and the model with two
discrete delays (black). The other parameters are $\alpha=-0.4,\tau_\alpha=1$.
The red line is the curve $\beta=1-\alpha$.
}
\label{comparefig}
\end{figure}

We also observe the following phenomena associated with a differential
equation with two discrete time delays \cite{Bela_Camp_1,Hale_Huang_1,Yuan_Campbell04}.
For appropriate parameter values, there is {\em stability switching}:
the trivial solution is asymptotically stable for $\tau_\beta$ small, but then destabilises
and re-stabilises a finite number of times as $\tau_\beta$ increases.
This can be seen in Figures~\ref{Otherparamfig} and \ref{comparefig}(b).
For $\tau_\alpha$ large enough, as depicted in Figure~\ref{Otherparamfig}, there can be
intersection points between different curves forming the boundary of the stability region.
These points correspond to parameter values where the characteristic equation has
either two pairs of pure imaginary roots (crossing of two black curves) or a pair
of pure imaginary roots and a zero root (crossing of black curve and red line).
If our characteristic equation were from the linearisation of a nonlinear delay
differential equation about an equilibrium point, such points would correspond to
points of {\em codimension two bifurcation} \cite{Bela_Camp_1}.

There are some differences, of course. In the discrete case, if
 $\tau_\alpha$ is small enough stability switching does not occur.
This can be seen in Figure~\ref{comparefig} (b). The curve which forms the
left boundary of the stability region is monotone increasing (thinking of
$\tau_\beta$ as a function of $\beta$) in the case of a discrete delay,
while the curves for the distributed delay are not monotone.
Further, when stability switching occurs in the discrete case, it ultimately
ends with the trivial solution being unstable \cite{Bela_Camp_1}.  For
the distributed case, our observation is that the stability switching
sometimes ends with the trivial solution being unstable (see
Figures~\ref{m345fig}  and \ref{comparefig}(b)). This is consistent with results of
\cite{Cook_Gross} for a system with one time delay. In general, we observe that the stability
region for the system with one discrete and one distributed time delay is {\em larger} than
that for two discrete time delays if other parameters are kept fixed. This can be seen in
Figure~\ref{comparefig}(b) and is consistent with the ``rule of thumb'' that a system with
a distributed time delay is inherently more stable than the corresponding system with a
discrete time delay. This is shown in many papers including \cite{Atay03a,Atay03b,Cook_Gross,CampbellJessop09,TSE}
and references therein.

\section*{Acknowledgements}
All plots were generated using the symbolic algebra package Maple. This
work has benefitted from the support of the Natural Sciences and Engineering
Research Council of Canada.

\end{document}